\documentclass[a4paper]{article}

\usepackage[margin=1in]{geometry}
\usepackage{graphicx}%
\usepackage{multirow}%
\usepackage{amsmath,amssymb,amsfonts}%
\usepackage{amsthm}%
\usepackage{mathrsfs}%
\usepackage[title]{appendix}%
\usepackage{xcolor}%
\usepackage{textcomp}%
\usepackage{manyfoot}%
\usepackage{booktabs}%
\usepackage{algorithm}%
\usepackage{algorithmicx}%
\usepackage{algpseudocode}%
\usepackage{listings}%
\usepackage{tikz-cd}
\usepackage{lineno}
\tikzcdset{every label/.append style = {font = \small}}
\usepackage{tikz}
\usetikzlibrary{graphs,decorations.pathmorphing,decorations.markings}
\usepackage{subcaption}
\usepackage{url}
\usepackage{hyperref}

\newtheorem{theorem}{Theorem}
\newtheorem*{theorem*}{Theorem}

\newtheorem{definition}{Definition}%

\newtheorem{lemma}[theorem]{Lemma}

\usetikzlibrary{graphs,decorations.pathmorphing,decorations.markings}
\usepackage{multirow}

\numberwithin{equation}{section}

\begin{document}

\title{The Effect of Vaccination on the Competitive Advantage of Two Strains of an Infectious Disease}

\author{%
  Matthew D. Johnston$^{*}$, 
  Bruce Pell, 
  Jared Pemberton, and  David A. Rubel \\ \\ Department of Mathematics + Computer Science\\Lawrence Technological University\\21000 W 10 Mile Rd, Southfield, MI 48075, USA\\ \tt{${}^{*}$mjohnsto1@ltu.edu}
}




\maketitle

\abstract{
We investigate how a population's natural and vaccine immunity affects the competitive balance between two strains of an infectious disease with different epidemiological characteristics. Specifically, we consider the case where one strain is more transmissible and the other strain is more immune-resistant. The competition of these strains is modeled by two SIR-type models which incorporate waning natural immunity and which have distinct mechanisms for vaccine immunity. Waning immunity is implemented as a gamma-distributed delay, which is analyzed using the linear chain trick to transform the delay differential equation systems into a system of ordinary differential equations. Our analysis shows that vaccination has a significant effect on the competitive balance between two strains, potentially leading to dramatic flips from one strain dominating in the population to the other. We also show that which strain gains an advantage as a population's immunity level increases depends upon the integration between the mechanisms of natural and vaccine immunity. The results of this paper are consequently relevant for public policy.
}


\section{Introduction}
\label{sec:introduction}
    
    The COVID-19 pandemic (2020-2023) and ensuing endemic phase has highlighted the importance of understanding the epidemiological factors by which new strains emerge, persist, co-exist, and dominate in a population. At any given time, there have been over 20 identified strains of COVID-19 circulating \cite{ourworldindataSARSCoV2Sequences}. This suggests that COVID-19 strains have a strong capacity for \emph{co-existence} in the global population. At the same time, however, we have seen two global variant takeover events (the Delta variant (B.1.617.2) takeover from 05/21-09/21 and the Omicron variant (B.1.1.529) takeover from 11/21-02/22). This suggests that there are critical epidemiological factors of COVID-19 variants which contribute to \emph{mutual exclusion} of strains.
    
    Narratives surrounding the emergence of COVID-19 variants have largely revolved around the increased transmissibility of emerging strains \cite{caCaliforniaDepartment}. Several recent mathematical modeling papers have suggested that, in competitions of two strains, mutually exclusion is the only realistic long-term possibility with the strictly more transmissible strain surviving while the others do not \cite{baba2018,kaymakamzade2018,wang2022,meskaf2020}. This mutual exclusion conclusion has also been reached in studies when vaccination is incorporated in the model \cite{ahumada2023,yaagoub2023-1}. However, factors which can contribute to co-existence of strains in mathematical models are historically known, and include co-infection, cross-immunity, seasonal fluctuations, and age-specific patterns of infectiousness \cite{white1998,martcheva2015}. In the context of COVID-19, modeling studies have illustrated instances of co-existence, driven by non-standard force of infection terms \cite{khyar2020, wang2022-3}, strain-specific vaccination strategies \cite{fudolig2020}, and delayed cross-immunity \cite{pell2023,Johnston2023}. Furthermore, recent cross-neutralization experiments on COVID-19 have suggested that the Omicron variant belongs to a different antigen cluster than prominent previous strains, such as Delta \cite{Kudriavtsev2022}. That paper hypothesizes that the resulting immune-resistance, rather than strict transmissibility, was responsible for Omicron's success in out-competing other strains in 2021 and 2022.
    
    In this paper, we address this gap in knowledge by investigating the relationship between transmissibility, vaccination, and immune-resistance. Specifically, we seek answers to the following two questions:
    \begin{enumerate}
        \item \emph{How does changing a population's effective immunity level through vaccination impact the competitive balance between two strains of an infectious disease?
        }
        \item \emph{How does a change in competitive balance between two strains depend upon the relationship between the mechanisms of the immune response following vaccination and previous infection?
        }
    \end{enumerate}
    Ultimately, we show the following results: (1) Changes in a population's immunity level can significantly alter the competitive balance between two strains, potentially catalyzing a shift from one strain's dominance to the other's; and (2) Disparities in the mechanisms of vaccine immunity and natural immunity can influence which strain gains a competitive advantage as a population's immunity level rises. When the two strains have disparate impacts on a population, for example when one strain is more virulent than the other, changing immunity levels in a population could have counter-intuitive effects on a population's well-being. These considerations are consequently important to public policy.

    We address these questions by considering two strains of an infectious disease, one which is more transmissible while the other is more immune-resistant. We introduce two SIR-type models of infectious disease spread which incorporate vaccination, immune-resistance, and waning immunity. In the first model, vaccine and natural immunity are \emph{integrated}; that is, the immune response after receiving a vaccine or recovering from an infection by either strain are identical. In the second model, vaccine and natural immunity are \emph{separated}, so that the immune responses are independent. We incorporate waning natural immunity as gamma-distributed delays which we transform into a system of ordinary differential equations through novel application of the `linear chain trick' \cite{Hurtado2019,Smith2010}. Gamma-distributed delays have been used significantly in recent years to model waning immunity, hospitalization time, and recovery times in infectious disease spread modeling \cite{ZHANG2019,Kamiya2022,Steindorf2022,Otunuga2022}. We show that there are four steady states (a disease-free state, a Strain 1-only endemic state, a Strain 2-only endemic state, and a co-existence endemic state) and compute the basic, strain-specific basic, and strain-specific competitive reproduction numbers. We conduct bifurcation analysis on the two strains with a focus on the parameters controlling the relative transmissibility and immune-resistance of the two strains, and the population's effective immunity level. These modeling results are consequently relevant for public policy as they suggest that vaccination influences strain-over-strain fitness.

\section{Basic Two Strain Models}
\label{sec:model}

In this section, we introduce two basic models of two-strain infectious disease spread with vaccination, which we expand upon in Section \ref{sec:delay}. In both models, the infection is broken down into two strains, which we label $I_1$ and $I_2$, and Strain 2 is assumed to have a greater degree of immune-resistance than Strain 1. Specifically, in the first model (the \emph{integrated immunity model}), Strain 2 has greater resistance to immunity provided by both previous infection by either strain and also to vaccination. In the second model (the \emph{separated immunity model}), Strain 2 has greater resistance to natural immunity from either strain but no greater immunity to the immunity provided by vaccination. The models we introduce are built upon the SIR (Susceptible-Infectious-Removed) models introduced by Kermack and McKendrick \cite{Kermack1927} and studied significantly since.

We consider multiple models with distinct mechanisms for vaccine immunity for two reasons: (1) to account for the incomplete and developing understanding of the relationship between vaccine and natural immunity, specifically within COVID-19 variants; and (2) to emphasize the starkly different responses infectious disease spread models can have based on different assumptions on the nature of this immunity. To the first point, studies have varied in their assessment of the relative immunity provided by vaccine and previous infection, with several randomly controlled studies concluding that natural immunity is stronger than vaccine immunity \cite{Gazit2021,Shenai2021} while other observational studies have concluded they are equivalent \cite{Townsend2022}. There are also recognized differences in the mechanisms by which vaccine and natural immunity manifest in the body, with vaccine immunity focusing on the production of antibodies, which are effective at fighting initial infection, while natural immunity also generates memory B-cells and cross-reaction T-cell responses, which are more responsive to fighting illness after infection and may be more robust to multiple strains \cite{Dolgin2021,Barouch2022}. To the second point, we will see in Section \ref{sec:example} that subtle changes to the nature of immunity can lead to wildly different predictions in terms of which strain will survive in a population.


We are particularly interested in how the distinct epidemiological characteristics of the strains factor into the competitive advantage for Strain 1 over Strain 2 or vice versa. In order to investigate this, we consider the basic reproduction number of the disease and the individual strains. The basic reproduction number of an infectious disease, denoted $\mathscr{R}_0$, is the expected number of new infections produced by a single infectious individual in a fully susceptible population \cite{Diekmann1990}. Consequently, the condition $\mathscr{R}_0 > 1$ indicates that the disease is likely to spread while the condition $\mathscr{R}_0 < 1$ indicates the disease is likely to decline. Similarly, strain-specific reproduction numbers can be defined. $\mathscr{R}_i$ is the basic reproduction number for Strain $i$ and a value $\mathscr{R}_i > 1$ indicates that Strain $i$ can infiltrate a disease-free population. $\mathscr{R}_{ij}$ is the basic reproduction number for Strain $i$ in the presence of Strain $j$. In this case, $\mathscr{R}_{ij} > 1$ indicates that Strain $i$ can infiltrate a population already infected at endemic levels by Strain $j$. We refer to $\mathscr{R}_0$ as the basic reproduction number, to $\mathscr{R}_i$ as the strain-specific basic reproduction numbers, and to $\mathscr{R}_{ij}$ as the strain-specific competitive reproduction numbers. Collectively, we refer to these values as reproduction numbers. For compartmental models of infectious disease spread, it is common to calculate the reproduction numbers using the next generation matrix method \cite{Diekmann1990,VANDENDRIESSCHE2002,Heffernan2005,diekmann2009,VANDENDRIESSCHE2017}. We present the details of this method in Appendix \ref{sec:2}.

With these reproduction numbers, we define the following distinct behaviors for two-strain infectious disease spread models.
\begin{definition}
\label{behavior}
Consider a two-strain infectious disease spread model with basic reproduction number $\mathscr{R}_0$, strain-specific basic reproduction numbers $\mathscr{R}_1$ and $\mathscr{R}_2$, and strain-specific competitive reproduction numbers $\mathscr{R}_{12}$ and $\mathscr{R}_{21}$. We say that the model exhibits:
\begin{enumerate}
    \item \textbf{disease-free behavior (DF)} if $\mathscr{R}_0 < 1$ (i.e. $\mathscr{R}_1 < 1$ and $\mathscr{R}_2 < 1$)
    \item \textbf{strain 1-only behavior (S1)} if 
    \begin{enumerate}
        \item $\mathscr{R}_1 > 1$ and $\mathscr{R}_2 < 1$, or 
        \item $\mathscr{R}_1 > 1$, $\mathscr{R}_2 > 1$, $\mathscr{R}_{12} > 1$, and $\mathscr{R}_{21} < 1$    
    \end{enumerate}
    \item \textbf{strain 2-only behavior (S2)} if 
    \begin{enumerate}
        \item $\mathscr{R}_1 < 1$ and $\mathscr{R}_2 > 1$, or 
        \item $\mathscr{R}_1 > 1$, $\mathscr{R}_2 > 1$, $\mathscr{R}_{12} < 1$, and $\mathscr{R}_{21} > 1$
    \end{enumerate}
    \item \textbf{co-existence behavior (C)} if $\mathscr{R}_1 > 1$, $\mathscr{R}_2 > 1$, $\mathscr{R}_{12} > 1$, and $\mathscr{R}_{21} > 1$.
\end{enumerate}
\end{definition}
We note that the transition of a basic reproduction number $\mathscr{R}_i$ from $\mathscr{R}_i<1$ to $\mathscr{R}_i>1$ implies the destabilization of the corresponding equilibrium but does not necessarily imply the stability of any other steady state. In particular, classifying the system as having co-existence behavior does not guarantee that the co-existence steady state is locally or globally asymptotically stable. Our numerical results suggest that local asymptotic stability is attained in all cases considered here; however, we leave a full investigation as future work.

\subsection{Basic Two-Strain Integrated Immunity Model}

The first model we consider has four compartments: Susceptibles ($S$), those infectious with Strain 1 ($I_1)$, those infectious with Strain 2 ($I_2$), and those who are temporarily immune from Strain 1 ($R$). We suppose that susceptible individuals are infected by those with Strain 1 at rate $\beta_1$ and those with Strain 2 at rate $\beta_2$. Individuals infected by either Strain 1 or Strain 2 recover at rate $\gamma$ and susceptible individuals vaccinate at rate $\lambda$. Individuals who have recently received an immune boost, either from natural or vaccine immunity, may be infected by the immune-resistant Strain 2 at rate $\beta_2$ or lose their temporary immunity to Strain 1 at rate $\alpha$. Note that, in this model, the immunity provided by vaccination is identical to that provided by a previous infection by either Strain 1 or Strain 2. 

These assumptions give the following model:
\begin{equation}
    \small
    \label{SIR1}
    \begin{tikzcd}
\mbox{\fbox{\begin{tabular}{c} Susceptible \\($S$)\end{tabular}}} \arrow[ddr,bend right=25,"\beta_1"'] \arrow[dr,"\beta_2"'] \arrow[rr,yshift=0.75ex,"\lambda"] & & \mbox{\fbox{\begin{tabular}{c} Partially \\ Immune \\($R$)\end{tabular}}} \arrow[ll,yshift=-0.75ex,"\alpha"] \arrow[dl, yshift=0.5ex,xshift=-0.5ex,"\beta_2"']\\[-0.25in]
 & \mbox{\fbox{\begin{tabular}{c} Strain 2 \\($I_2$)\end{tabular}}} \arrow[ur,yshift=-0.5ex,xshift=0.5ex,"\gamma"'] &  \\[-0.25in]
 & \mbox{\fbox{\begin{tabular}{c} Strain 1 \\($I_1$)\end{tabular}}} \arrow[uur,bend right=25,"\gamma"'] &
\end{tikzcd}
\end{equation}
The model \eqref{SIR1} corresponds to the following system of ordinary differential equations:
\begin{equation}
    \label{SIR1-DE}
    \left\{ \; \;
        \begin{aligned}
            \frac{dS}{dt} & = - \frac{\beta_1}{N} SI_1 - \frac{\beta_2}{N} SI_2 - \lambda S + \alpha R \\
            \frac{dI_1}{dt} & = \frac{\beta_1}{N} SI_1 - \gamma I_1 \\
            \frac{dI_2}{dt} & = \frac{\beta_2}{N} (S + R)I_2 - \gamma I_2 \\
            \frac{dR}{dt} & = \gamma (I_1 + I_2) - \frac{\beta_2}{N} RI_2 +\lambda S - \alpha R
        \end{aligned}
    \right.
\end{equation}
where the parameters are as in Table \ref{table1} and there is the population conservation relation $N = S + R + I_1 + I_2$. 
For the model \eqref{SIR1-DE}, we have the following values:
\begin{equation}
    \label{R01}
\begin{aligned}
& \mathscr{R}_1 = \frac{\beta_1}{\gamma} \left( \frac{\alpha }{\alpha+ \lambda}\right), \mathscr{R}_2 = \frac{\beta_2}{\gamma}, \mathscr{R}_0 = \max \{ \mathscr{R}_1, \mathscr{R}_2 \},\\
& \mathscr{R}_{12} = \frac{\beta_1}{\beta_2} \left( \frac{\alpha}{ \alpha + \beta_2 + \lambda -\gamma}\right), \mathscr{R}_{21} = \frac{\beta_2}{\beta_1} \left( \frac{\beta_1 + \alpha + \lambda}{\gamma + \alpha} \right).
\end{aligned}
\end{equation}
These basic reproduction number values result from special cases of the computations provided in Appendix \ref{sec:21}. We consider bifurcations, steady state diagrams, and numerical simulations of \eqref{SIR1-DE} in Section \ref{sec:example}.

\begin{table}[t!]
    \centering
    \begin{tabular}{l|c|l}
    \hline \hline
    Variable & Units & Description \\
    \hline \hline
    $S, S_0 \geq 0$ & people & Susceptible individuals \\
    $R, S_i \geq 0$ & people & Partially immune individuals ($i=1,\ldots, k$ levels) \\
    $I_1 \geq 0$ & people & Infectious individuals (Strain $1$) \\
    $I_2 \geq 0$ & people & Infectious individuals (Strain $2$) \\
    $V \geq 0$ & people & Vaccinated individuals \\
    $t \geq 0$ & days & Time elapsed \\
    \hline \hline
    Parameter & Units & Description  \\
    \hline \hline
    $N \geq 0$ & people & Population size \\
    $\beta_1 \geq 0$ & days$^{-1}$ & Transmission rate (Strain $1$) \\
    $\beta_2 \geq 0$ & days$^{-1}$ & Transmission rate (Strain $2$) \\
    $\lambda \geq 0 $ & days$^{-1}$ & Vaccination rate \\
    $\gamma^{-1} > 0$ & days & Mean infectious period\\
    $\alpha^{-1} > 0$ & days & Temporary immunity period from Strain 1\\
    $\left( 1 - \frac{r}{k} \right)\alpha^{-1} \geq 0$ & days & Temporary immunity period from Strain 2\\
    $0 \leq \epsilon \leq 1$ & --- & Population's effective vaccination proportion \\
    $k \geq 0$ & --- & Number of partially immunity states \\
    $0 \leq r \leq k$ & --- & Degree of Strain 2 partial immune resistance\\
    $\mathscr{R}_0 \geq 0$ & --- & Basic reproduction number of disease \\
    $\mathscr{R}_1 \geq 0$ & --- & Basic reproduction number of Strain $1$ \\
    $\mathscr{R}_2 \geq 0$ & --- & Basic reproduction number of Strain $2$ \\
    $\mathscr{R}_{12} \geq 0$ & --- & Basic reproduction number of Strain $1$ in presence of Strain $2$ \\
    $\mathscr{R}_{21} \geq 0$ & --- & Basic reproduction number of Strain $2$ in presence of Strain $1$ \\
    \hline \hline
    \end{tabular}
    \caption{\small Variables and parameters for the two-strain models: \eqref{SIR1}, \eqref{SIR2}, \eqref{SIR3}, \eqref{SIR4}, \eqref{SIR5}, and \eqref{SIR6}.}
    \label{table1}
\end{table}

\subsection{Basic Two-Strain Separated Immunity Model}

We now consider a model where vaccination is incorporated but the immunity provided by vaccination is separated from natural immunity. We introduce a compartment for vaccinated individuals, $V$, and allow susceptible and partially immune individuals to vaccinate at the rate $\lambda$ into compartment $V$, and vaccinated individuals to lose their immunity at rate $\alpha$.

These assumptions give rise to the following model:
\begin{equation}
    \small
    \label{SIR2}
    \begin{tikzcd}
    & \mbox{\fbox{\begin{tabular}{c} Vaccinated \\($V$)\end{tabular}}} \arrow[dl,yshift=-0.5ex,xshift=0.5ex,"\alpha"] & \\[-0.25in]
    \mbox{\fbox{\begin{tabular}{c} Susceptible \\($S$)\end{tabular}}} \arrow[ddr,bend right=25,"\beta_1"'] \arrow[dr,"\beta_2"] \arrow[ur,yshift=0.5ex,xshift=-0.5ex,"\lambda"] & & \mbox{\fbox{\begin{tabular}{c} Partially \\ Immune \\($R$)\end{tabular}}} \arrow[ll,"\alpha"] \arrow[ul,"\lambda"'] \arrow[dl, yshift=0.5ex,xshift=-0.5ex,"\beta_2"']\\[-0.25in]
 & \mbox{\fbox{\begin{tabular}{c} Strain 2 \\($I_2$)\end{tabular}}} \arrow[ur,yshift=-0.5ex,xshift=0.5ex,"\gamma"'] & \\[-0.25in]
 & \mbox{\fbox{\begin{tabular}{c} Strain 1 \\($I_1$)\end{tabular}}} \arrow[uur,bend right=25,"\gamma"']  &
\end{tikzcd}
\end{equation}
Notice that, unlike in the integrated immunity model \eqref{SIR1}, in the separated immunity model \eqref{SIR2} vaccinated individuals may not be infected by Strain 2 before their immunity wanes. Consequently, Strain 2 is resistant to natural immunity but not to vaccine immunity. To account for this disparity between Strain 1 and Strain 2, we have allowed partially immune individuals to vaccinate since they are still susceptible to Strain 2 otherwise.

The model \eqref{SIR2} corresponds to the following system of ordinary differential equations:
\begin{equation}
    \label{SIR2-DE}
    \left\{ \; \;
        \begin{aligned}
            \frac{dS}{dt} & = - \frac{\beta_1}{N} SI_1 - \frac{\beta_2}{N} SI_2 - \lambda S + \alpha R + \alpha V \\
            \frac{dI_1}{dt} & = \frac{\beta_1}{N} SI_1 - \gamma I_1 \\
            \frac{dI_2}{dt} & = \frac{\beta_2}{N} (S + R)I_2 - \gamma I_2 \\
            \frac{dR}{dt} & = \gamma (I_1 + I_2) - \frac{\beta_2}{N} RI_2 - (\lambda + \alpha) R\\
            \frac{dV}{dt} & = \lambda (S + R) - \alpha V
        \end{aligned}
    \right.
\end{equation}
where the parameters are as in \eqref{SIR1-DE} (see Table \ref{table1}) and there is the conservation law $N = S + R + I_1 + I_2 + V$. We have the following reproduction numbers:
\begin{equation}
    \label{R02}
\begin{aligned}
& \mathscr{R}_1 = \frac{\beta_1}{\gamma} \left( \frac{ \alpha }{\alpha+ \lambda} \right), \mathscr{R}_2 = \frac{\beta_2}{\gamma} \left( \frac{\alpha}{\alpha + \lambda} \right), \mathscr{R}_0 = \max \{ \mathscr{R}_1, \mathscr{R}_2 \},\\
& \mathscr{R}_{12} = \frac{\beta_1}{\beta_2} \left( \frac{ \alpha (\lambda + \alpha)}{\alpha (\beta_2 - \gamma + \lambda + \alpha)- \gamma \lambda} \right), \mathscr{R}_{21} = \frac{\beta_2}{\beta_1} \left( \frac{ \alpha ( \alpha + \beta_1 + \lambda)}{ (\lambda + \alpha)(\gamma + \alpha)}\right).\\
\end{aligned}
\end{equation}
These basic reproduction number values result from special cases of the computations provided in Appendix \ref{sec:22}. We consider bifurcations, steady state diagrams, and numerical simulations of \eqref{SIR2-DE} in Section \ref{sec:example}.



\subsection{Example}
\label{sec:example}

We now consider the effect of vaccination on the competitive balance between Strain 1 and Strain 2 
by considering specific instances of the two-strain integrated immunity model \eqref{SIR1} and the two-strain separated immunity model \eqref{SIR2}.

\begin{figure}[t!]
\centering
     \begin{subfigure}[h]{0.45\textwidth}
         \centering
         \includegraphics[width=\textwidth]{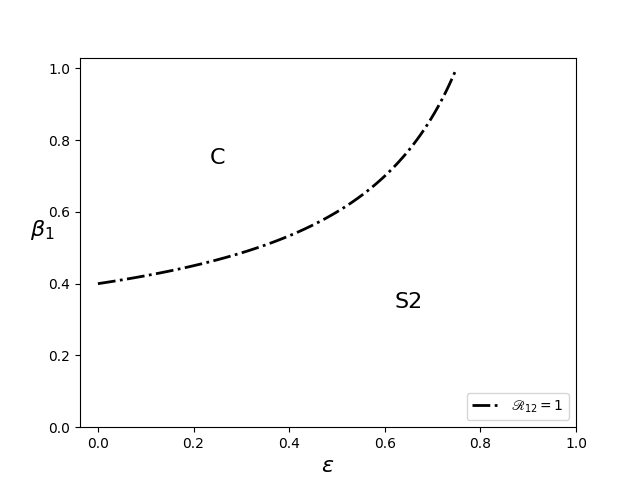}
         \caption{\footnotesize Bifurcation diagram for basic integrated immunity model \eqref{SIR1}}
     \end{subfigure}
     \hfill
     \begin{subfigure}[h]{0.45\textwidth}
         \centering
         \includegraphics[width=\textwidth]{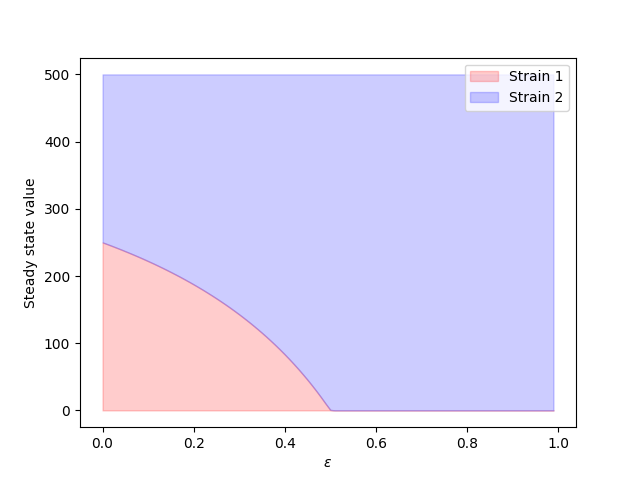}
         \caption{\footnotesize Steady state plot for basic integrated immunity model \eqref{SIR1}}
     \end{subfigure}
     \newline
     \begin{subfigure}[h]{0.45\textwidth}
         \centering
         \includegraphics[width=\textwidth]{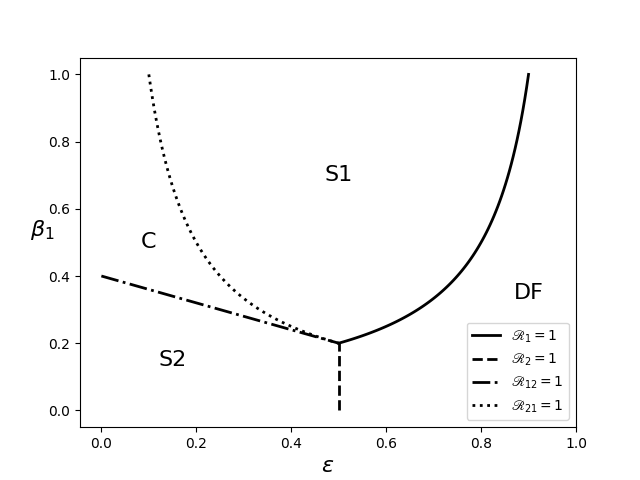}
         \caption{\footnotesize Bifurcation diagram for basic separated immunity model \eqref{SIR2}}
     \end{subfigure}
     \hfill
     \begin{subfigure}[h]{0.45\textwidth}
         \centering
         \includegraphics[width=\textwidth]{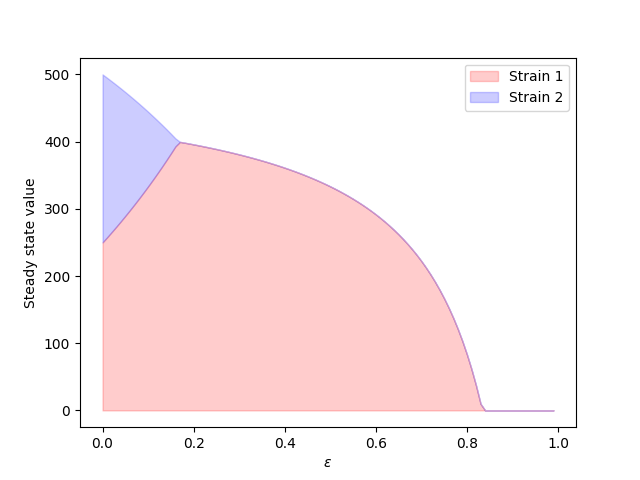}
         \caption{\footnotesize Steady state plot for basic separated immunity model \eqref{SIR2}}
     \end{subfigure}
\caption{\footnotesize In (a) and (c), we present bifurcation plots of the basic integrated immunity model \eqref{SIR1} and basic separated immunity model \eqref{SIR2}, respectively, over $0 \leq \beta_1 \leq 1$ and $0 \leq \epsilon \leq 1$ for the values $\beta_2 = 0.2$, $\gamma = 0.1$, and $\alpha = 0.1$. The four qualitatively distinct regions in Definition \ref{behavior} are given: DF - disease-free behavior; S1 - strain 1 only behavior; S2 - strain 2 only behavior; and C - co-existence behavior. In (b) and (d), we present plots of the stable steady state values where the total infection ($I_1+I_2$, upper curve) is broken down by Strain 1 ($I_1$, shaded red) and Strain 2 ($I_2$, shaded blue) as a function of the effective vaccination level of the population ($\epsilon$). The parameter values used for the simulation are $\beta_1 = 0.6$, $\beta_2 = 0.2$, $\gamma=0.1$, $\alpha = 0.1$, and $N=1000$, with initial conditions $S(0) = 980$, $I_1(0)=I_0(0)=10$, $R(0)=0$, and $V(0)=0$ (if necessary) and the ending condition $t=10000$. Note that this corresponds to $\beta_1 = 0.6$ in the bifurcation diagrams (a) and (c). The integrated immunity model \eqref{SIR1} exhibits \emph{co-existence behavior} for $0 \leq \epsilon < \frac{1}{2}$ and \emph{Strain 2 only behavior} for $\frac{1}{2} < \epsilon \leq 1$. The separated immunity model \eqref{SIR2} exhibits \emph{co-existence behavior} for $0 \leq \epsilon < \frac{1}{6}$, \emph{Strain 1 only behavior} for $\frac{1}{6} < \epsilon \leq \frac{5}{6}$, and \emph{disease-free behavior} for $\frac{5}{6} < \epsilon \leq 1$. Note that, in (b), the steady state value for the total infection stays constant at the value of 500 individuals as the effective immunity level rises while, in (d), the total infection decreases. Consequently, vaccination is able to eliminate the disease in the separated model \eqref{SIR2} but only eliminate Strain 1 in the integrated model \eqref{SIR1}.}
\label{fig:1}
\end{figure}

To aid in interpreting the parameter $\lambda$, we consider the systems \eqref{SIR1-DE} and \eqref{SIR2-DE} at the disease-free equilibrium. For \eqref{SIR1-DE} this is given by $S = \frac{\alpha N}{\alpha + \lambda}, R = \frac{\lambda N}{\alpha + \lambda}$ and the rest of the variables equal to zero, while for \eqref{SIR2-DE} it is $S = \frac{\alpha N}{\alpha + \lambda}, V = \frac{\lambda N}{\alpha + \lambda}$ and the rest of the variables equal to zero. In both cases, this justifies introducing a new parameter
\begin{equation}
    \label{epsilon}
\epsilon = \frac{\lambda}{\alpha + \lambda}
\end{equation}
which corresponds to the proportion of the population which is effectively immune to Strain 1 at the disease-free equilibrium. Note that $0 \leq \epsilon \leq 1$ and $\epsilon$ rises monotonically as $\lambda$ rises. The variable $\epsilon$ can consequently act as an effective proxy for $\lambda$. We incorporate \eqref{epsilon} in the models \eqref{SIR1-DE} and \eqref{SIR2-DE} by substituting $\lambda = \frac{\epsilon \alpha}{1 - \epsilon}$.

In Figure \ref{fig:1}, we present bifurcation diagrams and steady state diagrams for the integrated immunity model \eqref{SIR1} and separated immunity model \eqref{SIR2}. The bifurcation diagrams are generated over the regions $0 \leq \beta_1 \leq 1$ and $0 \leq \epsilon \leq 1$ for the parameter values $\beta_2 = 0.2$, $\gamma = 0.1$, and $\alpha = 0.1$. The steady states are numerically computed for fixed values of $0 \leq \epsilon \leq 1$ using the parameter values $\beta_1 = 0.6$, $\beta_2 = 0.2$, $\gamma = 0.1$, $\alpha = 0.1$, and $N=1000$ and initial conditions $S(0) = 980$, $I_1(0)=I_0(0)=10$, $R(0)=0$, and $V(0)=0$ (if necessary) and the ending condition $t=10000$.


To further understand the transitions on the strain-by-strain breakdown of total infection in the integrated immunity model \eqref{SIR1} (Figure \eqref{fig:1}(b)), we substitute these parameter values and $\lambda = \frac{\epsilon \alpha}{1 - \epsilon}$ into the reproduction numbers \eqref{R01} to get
\[
\begin{aligned}
& \mathscr{R}_1 = 6(1-\epsilon), \mathscr{R}_2 = 2, \mathscr{R}_0 = \max \{ 6(1-\epsilon), 2 \}, \mathscr{R}_{12} = \frac{3(1-\epsilon)}{2-\epsilon}, \mathscr{R}_{21} = \frac{7-6\epsilon}{6(1-\epsilon)}.
\end{aligned}
\]
It can be derived that $\mathscr{R}_1 >1$, $\mathscr{R}_2 > 1$, $\mathscr{R}_{12} > 1$, and $\mathscr{R}_{21} > 1$ for $0 \leq \epsilon < \frac{1}{2}$ while $\mathscr{R}_{12} < 1$ and $\mathscr{R}_{21} > 1$ for $\frac{1}{2} < \epsilon \leq 1$. Consequently, the system exhibits \emph{co-existence behavior} when $0 \leq \epsilon < \frac{1}{2}$ and \emph{Strain 2 only behavior} for $\frac{1}{2} < \epsilon \leq 1$.

For the separated immunity model \eqref{SIR1} (Figure \eqref{fig:1}(d)), we substitute the same parameter values into the reproduction numbers \eqref{R02} to get
\[ \small
\begin{aligned}
& \mathscr{R}_1 = 6(1-\epsilon), \mathscr{R}_2 = 2(1-\epsilon), \mathscr{R}_0 = \max \{ 6(1-\epsilon), 2(1-\epsilon) \}, \mathscr{R}_{12} = \frac{3}{2(1-\epsilon)}, \mathscr{R}_{21} = \frac{7 - 6\epsilon}{6}.
\end{aligned}
\]
It can be derived that $\mathscr{R}_1 >1$, $\mathscr{R}_2 > 1$, $\mathscr{R}_{12} > 1$, and $\mathscr{R}_{21} > 1$ for $0 \leq \epsilon < \frac{1}{6}$, $\mathscr{R}_{12} > 1$ and $\mathscr{R}_{21} < 1$ for $\frac{1}{6} < \epsilon \leq \frac{5}{6}$, and $\mathscr{R}_0 = \max \{ 6(1-\epsilon), 2(1-\epsilon) \} < 1$ for $\frac{5}{6} < \epsilon \leq 1$. Consequently, the system exhibits \emph{co-existence behavior} for $0 \leq \epsilon < \frac{1}{6}$, \emph{Strain 1 only behavior} for $\frac{1}{6} < \epsilon \leq \frac{5}{6}$, and \emph{disease-free behavior} for $\frac{5}{6} < \epsilon \leq 1$. 

We observe that that changes in a population's vaccination level can significantly alter the competitive balance between two competing strains of an infection, and that the nature of the immunity provided by vaccination matters in which strain is given an advantage. When the immunity provided by vaccination is integrated with the natural immunity \eqref{SIR1}, increasing vaccination favors the more immune-resistance strain ($I_2$). When the immunity provided by vaccination is separated from natural immunity \eqref{SIR2}, however, increasing vaccination favors the more transmissible strain ($I_1$). To investigate the capacity for vaccination to influence the competitive balance between two strains further, in Section \ref{sec:delay} we extend the models \eqref{SIR1} and \eqref{SIR2} to incorporate more realistic waning and partial immunity.

\section{Two-Strain Models with Partial Immune Resistance}
\label{sec:delay}

In the basic models \eqref{SIR1} and \eqref{SIR2}, we assumed that individuals passes directly from a contagious state (either $I_1$ or $I_2$) to a state where they are susceptible to Strain 2 ($R$), and then directly to a state where they are susceptible to both Strain 1 and Strain 2 ($S$). In reality, however, there is some time after infection where individuals are not contagious but not susceptible to infection. Furthermore, studies on COVID-19 have shown that an individual's risk of infection depends on the length of time since vaccination or previous infection and that this time may differ by strain \cite{Dolgin2021,Goldberg2022,PIECHOTTA2022}.


In order to incorporate stages of the strain-specific waning immunity periods into \eqref{SIR1} and \eqref{SIR2}, we imagine that there is a delay after previous infection to the partially immune state where Strain 2 can infect an individual, and another delay from this state to the fully susceptible state where both Strain 1 and 2 can infect an individual. We incorporate these delays as gamma-distributed integral delays. Recall that the gamma distribution with integer shape parameter $b$ and rate parameter $a$ has the form of an Erlang distribution:
\[g^b_a(t) = \frac{a^bt^{b-1}e^{-at}}{(b-1)!}\]
which has mean $\mu = \frac{b}{a}$. We will represent integral time-delayed edges with squiggly, rather than solid, arrows and label with the outgoing edge and delayed incoming distribution. For example, to add a $g^a_b(t)$ distributed delay to the transition
\[
\begin{tikzcd}
A \arrow[rr,"\alpha"] & & B
\end{tikzcd}
\]
we write
\[
\begin{tikzcd}
A \arrow[rr,rightsquigarrow,"\alpha\; (g_a^b)"] & & B.
\end{tikzcd}
\]
This indicates that individuals leave the state $A$ at rate $\alpha$ and enter state $B$ at a $g_a^b(t)$ distributed time later. Analysis of delay-differential equations is notoriously challenging; however, for models with gamma-distributed integral delays, the linear chain trick may be utilized to produce a set of linear ordinary differential equations in a set of new intermediate variables \cite{Smith2010,Hurtado2019}. We provide the details of our application of the linear chain trick in Appendix \ref{sec:3}.

\subsection{Two Strain Integrated Vaccination Model with Delays}


We now update the two-strain integrated and separated immunity vaccination models, \eqref{SIR1} and \eqref{SIR2}, to include partial immune resistance for Strain 2. To accomplish this, we incorporate gamma-distributed delays coming into and out of the partially immune state $R$ to get the following model:

\begin{equation}
    \small
    \label{SIR3}
    \begin{tikzcd}
\mbox{\fbox{\begin{tabular}{c} Susceptible \\($S$)\end{tabular}}} \arrow[ddr,bend right=25,"\beta_1"'] \arrow[dr,"\beta_2"'] \arrow[rr,yshift=0.75ex,rightsquigarrow,"\lambda \; (g_{k\alpha}^{k-r})"] & & \mbox{\fbox{\begin{tabular}{c} Partially \\ Immune \\($R$)\end{tabular}}} \arrow[ll,yshift=-0.75ex,rightsquigarrow,"k\alpha \; (g_{k\alpha}^r)"] \arrow[loop right,rightsquigarrow,pos=0.15,"\lambda \; (g_{k\alpha}^{k-r})"] \arrow[dl, yshift=0.5ex,xshift=-0.5ex,"\beta_2"']\\
 & \mbox{\fbox{\begin{tabular}{c} Strain 2 \\($I_2$)\end{tabular}}} \arrow[ur,yshift=-0.5ex,xshift=0.5ex,rightsquigarrow,xshift=0.5ex,pos=0.1,"\gamma \; (g_{k\alpha}^{k-r})"'] & \\
 & \mbox{\fbox{\begin{tabular}{c} Strain 1 \\($I_1$)\end{tabular}}} \arrow[uur,bend right=25,rightsquigarrow,xshift=0.5ex,"\gamma \; (g_{k\alpha}^{k-r})"'] &
\end{tikzcd}
\end{equation}
where the squiggly arrows indicate edges which are time-delayed. Note that the delays in \eqref{SIR3} are structured with shape parameter $k-r$ and rate parameter $k\alpha$ into state $R$ and shape parameter $r$ and rate parameter $k\alpha$ into state $S$. This implies that individuals wait an average time of $\mu_1 = \frac{k-r}{k\alpha} = (1-\frac{r}{k})\alpha^{-1}$ after leaving the infectious classes to enter the partially immune state, and an average time of $\mu_2 = \frac{r}{k\alpha} = \frac{r}{k} \alpha^{-1}$ after leaving the partially immune state to enter the susceptible class. This gives a total mean time of $\mu_1 + \mu_2 = (1-\frac{r}{k})\alpha^{-1} + \frac{r}{k} \alpha^{-1} = \alpha^{-1}$ after leaving the infectious classes to become fully susceptible again. This is the same expected time as the basic model \eqref{SIR1} which had no delays.

Note that individuals who are transitioning from $R$ to $S$ along the delayed edge are not removed from the population; rather, they are still susceptible to infection from Strain $2$ and therefore may transition to $I_2$ before reaching $S$. Consequently, those ultimately transitioning to $S$ will be reduced by the proportion who become infected during the delay period. We assume that the proportion of individuals who transition from $R$ to $I_2$ instead of $S$ is exponentially distributed with time-dependent rate parameter $\frac{\beta_2}{N} I_2(t)$.

These assumptions give rise to the following system of distributed-delay differential equations:
\begin{equation}
    \left\{ \; \; \;
    \label{SIR3-DE}
    \begin{aligned}
        \frac{dS}{dt} & = - \frac{\beta_1}{N} S I_1 - \frac{\beta_2}{N} S I_2 - \lambda S + \Pi(t)\\
        \frac{dI_1}{dt} & = \frac{\beta_1}{N} SI_1 - \gamma I_1\\
        \frac{dI_2}{dt} & = \frac{\beta_2}{N} (S+R)I_2 - \gamma I_2\\
        \frac{dR}{dt} & = \Gamma(t) - \frac{\beta_2}{N} R I_2 - \lambda R - \Pi(t)\\
    \end{aligned}
    \right.
\end{equation}
where
\begin{equation}
    \label{Gamma}
    \begin{aligned}
    \Gamma(t) & = \int_{-\infty}^t [\lambda (S(\tau) + R(\tau)) + \gamma(I_1(\tau) + I_2(\tau))] g_{k\alpha}^{k-r}(t-\tau)\; d\tau,\\
    \Pi(t) & = \int_{-\infty}^t \Gamma(\tau) e^{-\frac{\beta_2}{N} \int_{\tau}^t I_2(s) \; ds} e^{-\lambda(t-\tau)} g_{k\alpha}^r(t- \tau) \; d\tau.
    \end{aligned}
\end{equation}


The system of delay differential equations \eqref{SIR3-DE} with \eqref{Gamma} is challenging to analyzing directly.  Applying the linear chain trick to \eqref{SIR3}, however, we arrive at the following equivalent delay-free model:
\begin{equation}
    \tiny
    \label{SIR4}
    \begin{tikzcd}
\mbox{\fbox{\begin{tabular}{c} Susceptible \\($S_0$)\end{tabular}}} \arrow[ddrrr,bend right=25,"\beta_1"' font=\tiny] \arrow[drrr,bend right = 15,"\beta_2"' font=\tiny] \arrow[rrrrrr,bend left = 30,"\lambda" font=\tiny] & \mbox{\fbox{\begin{tabular}{c} Partially \\ Immune \\($S_1$)\end{tabular}}}  \arrow[l,"k\alpha"' font=\tiny] \arrow[drr,bend right = 10,"\beta_2"' font=\tiny] \arrow[rrrrr,bend left = 25,"\lambda" font=\tiny] & \cdots \arrow[l,"k\alpha"' font=\tiny] & \mbox{\fbox{\begin{tabular}{c} Partially \\ Immune \\($S_r$)\end{tabular}}}  \arrow[l,"k\alpha"' font=\tiny] \arrow[d, "\beta_2" font=\tiny] \arrow[rrr,bend left = 20,"\lambda" font=\tiny] & \cdots \arrow[l,"k\alpha"' font=\tiny] & \mbox{\fbox{\begin{tabular}{c} Partially \\ Immune \\($S_{k-1}$)\end{tabular}}} \arrow[l,"k\alpha"' font=\tiny] & \mbox{\fbox{\begin{tabular}{c} Partially \\ Immune \\($S_k$)\end{tabular}}} \arrow[l,"k\alpha"' font=\tiny] \\
 & & & \mbox{\fbox{\begin{tabular}{c} Strain 2 \\($I_2$)\end{tabular}}} \arrow[urrr,bend right=15,"\gamma"' font=\tiny] & & & \\
 & & & \mbox{\fbox{\begin{tabular}{c} Strain 1 \\($I_1$)\end{tabular}}} \arrow[uurrr,bend right = 25,"\gamma"' font=\tiny] & & &
\end{tikzcd}
\end{equation}
where $0 \leq r \leq k$. Note that the model \eqref{SIR4} contains \eqref{SIR1} as a special case when $r=k=1$. This system can be corresponded to the following system of ordinary differential equations:
\begin{equation}
\label{SIR4-DE}
    \left\{ \; \; \;
    \begin{aligned}
        \frac{dS_0}{dt} & = -\frac{\beta_1}{N} S_0 I_1 -\frac{\beta_2}{N} S_0 I_2 + k\alpha S_1 - \lambda S_0 &&  \\
        \frac{dS_i}{dt} & =  -\frac{\beta_2}{N} S_i I_2 + k\alpha {S}_{i+1} - k\alpha S_i - \lambda S_i, \; \; \; \; \; \; \; \; \; && i = 1, \ldots, r \\
        \frac{dS_i}{dt} & = k\alpha {S}_{i+1} - k\alpha S_i, && i = r+1, \ldots, k-1 \\
        \frac{dS_k}{dt} & = \gamma (I_1+ I_2) - k\alpha S_k + \lambda \sum_{i=0}^r S_i && \\
        \frac{dI_1}{dt} & = \frac{\beta_1}{N} S_0 I_1 - \gamma I_1&& \\
        \frac{dI_2}{dt} & = \frac{\beta_2}{N} I_2 \sum_{i=0}^r S_i - \gamma I_2.&&
    \end{aligned}
    \right.
\end{equation}

In Appendix \ref{sec:31}, we prove the following.
\begin{theorem}
    \label{thm:3}
    The system of distributed delay differential equations \eqref{SIR3-DE} is equivalent to the system of ordinary differential equations \eqref{SIR4-DE} with the following variable substitutions:
\[
    \begin{aligned}
        S_0(t) & = S(t), \\
        S_i(t) & = \frac{1}{k\alpha} \int_{-\infty}^t \Gamma(\tau) e^{-\frac{\beta_2}{N} \int_{\tau}^t I_2(s) \; ds} e^{-\lambda(t-\tau)} g_{k\alpha}^{r-i+1}(t-\tau) \; d\tau, & & i = 1, \ldots, r\\
        S_i(t)  & = \frac{1}{k\alpha} \int_{-\infty}^t [\lambda (S(\tau)+R(\tau)) + \gamma(I_1(\tau)+I_2(\tau))]g_{k\alpha}^{k-i+1}(t-\tau)\; d\tau, & & i = r+1, \ldots, k.\\
    \end{aligned}
\]
where $R(t) = \sum_{i=1}^r S_i(t)$.
\end{theorem}

In Section 2.1 of the Supplementary Material, we prove that the system \eqref{SIR4-DE} has the following reproduction numbers:
\begin{equation} \small
    \label{R05}
    \begin{aligned}
    \mathscr{R}_1 & = \frac{\beta_1}{\gamma} \left(\frac{k\alpha}{k\alpha + (k - r)\lambda}\right) \left(\frac{k\alpha}{k\alpha + \lambda}\right)^r, \mathscr{R}_2 = \frac{\beta_2}{\gamma} \left(\frac{k\alpha}{k\alpha + (k - r)\lambda}\right), \mathscr{R}_0 = \max \{ \mathscr{R}_1, \mathscr{R}_2 \},\\
    \mathscr{R}_{12} & = \frac{\beta_1}{\beta_2} \left( \frac{(k - r)\gamma + k \alpha}{(k - r - 1) \gamma + k \alpha + \beta_2 + \lambda} \right)^r,\\
    \mathscr{R}_{21} & = \frac{\beta_2}{\beta_1} \left( \frac{ (\gamma\lambda (k - r) + (\beta_1 + \lambda) k\alpha) (k\alpha + \lambda)^r  - \beta_1 (k\alpha)^{r+1} }{ \gamma (k\alpha + (k - r)\lambda) (k\alpha + \lambda)^r + (\lambda - \gamma) (k\alpha)^{r+1} } \right).
    \end{aligned}
\end{equation}



\subsection{Two Strain Separated Vaccination Model with Delays}

We can also introduce partial immunity through integrated delays to the basic two-strain with separate vaccine immunity model \eqref{SIR2}. This gives the following model:

\begin{equation}
    \small
    \label{SIR5}
    \begin{tikzcd}
    & \mbox{\fbox{\begin{tabular}{c} Vaccinated \\($V$)\end{tabular}}} \arrow[dl,yshift=-0.5ex,xshift=0.5ex,"\alpha"] & \\
    \mbox{\fbox{\begin{tabular}{c} Susceptible \\($S$)\end{tabular}}} \arrow[ddr,bend right=25,"\beta_1"'] \arrow[dr,"\beta_2"] \arrow[ur,yshift=0.5ex,xshift=-0.5ex,"\lambda"] & & \mbox{\fbox{\begin{tabular}{c} Partially \\ Immune \\($R$)\end{tabular}}} \arrow[ll,rightsquigarrow,"k\alpha \; (g_{k\alpha}^{r})"] \arrow[ul,"\lambda"'] \arrow[dl, yshift=0.5ex,xshift=-0.5ex,"\beta_2"']\\
 & \mbox{\fbox{\begin{tabular}{c} Strain 2 \\($I_2$)\end{tabular}}} \arrow[ur,yshift=-0.5ex,xshift=0.5ex,pos=0.1,rightsquigarrow,"\gamma \; (g_{k\alpha}^{k-r})"'] & \\
 & \mbox{\fbox{\begin{tabular}{c} Strain 1 \\($I_1$)\end{tabular}}} \arrow[uur,bend right=25,rightsquigarrow,"\gamma \; (g_{k\alpha}^{k-r})"']  &
\end{tikzcd}
\end{equation}

Notice that, for the model \eqref{SIR5}, after entering the temporarily immune state by either vaccination (i.e. entering the class $V$) or recovering from infection (i.e. entering the class $S_k$), it takes an average of $\alpha^{-1}$ days to return to complete susceptibility (i.e. the class $S_0$). After recovering from infection, however, individuals become susceptible to the immune-resistant Strain 2 on average after $\left(1 - \frac{r}{k} \right)\alpha^{-1}$ days. Also notice that we have left the transition from the Vaccinated class ($V$) to the Susceptible class ($S$) as exponentially-distributed, rather than gamma-distributed. We have kept this transition for simplicity but note that, since the expected waning immunity time remains the same in both cases, the steady states and consequently the basic and strain-specific reproduction numbers are not affected.

The model \eqref{SIR5} can be corresponded to the following system of delay differential equations:
\begin{equation}
    \left\{ \; \; \;
    \label{SIR5-DE}
    \begin{aligned}
        \frac{dS}{dt} & = - \frac{\beta_1}{N} S I_1 - \frac{\beta_2}{N} S I_2 - \lambda S + \alpha V + \Pi(t)\\
        \frac{dI_1}{dt} & = \frac{\beta_1}{N} SI_1 - \gamma I_1\\
        \frac{dI_2}{dt} & = \frac{\beta_2}{N} (S+R)I_2 - \gamma I_2\\
        \frac{dR}{dt} & = \Gamma(t) - \frac{\beta_2}{N} R I_2 - \lambda R - \Pi(t)\\
        \frac{dV}{dt} & = \lambda (S + R) - \alpha V
    \end{aligned}
    \right.
\end{equation}
where
\begin{equation}
    \label{Gamma2}
    \begin{aligned}
    \Gamma(t) & = \int_{-\infty}^t \gamma(I_1(\tau) + I_2(\tau)) g_{k\alpha}^{k-r}(t-\tau)\; d\tau,\\
    \Pi(t) & = \int_{-\infty}^t \Gamma(\tau) e^{-\frac{\beta_2}{N} \int_{\tau}^t I_2(s) \; ds} e^{-\lambda(t-\tau)} g_{k\alpha}^r(t- \tau) \; d\tau.
    \end{aligned}
\end{equation}

We can apply linear chain trick to \eqref{SIR5-DE} to expand the gamma-distributed integral delays \eqref{Gamma2} into linear exponentially-distributed delays to get the following model:

\begin{equation}
    \tiny
    \label{SIR6}
    \begin{tikzcd}
    & & & \mbox{\fbox{\begin{tabular}{c} Vaccinated \\($V$)\end{tabular}}} \arrow[dlll,bend right = 15, yshift=1ex,"\alpha"' font=\tiny] & & & \\
\mbox{\fbox{\begin{tabular}{c} Susceptible \\($S_0$)\end{tabular}}} \arrow[ddrrr,bend right=25,"\beta_1"' font=\tiny] \arrow[drrr,bend right = 15,"\beta_2"' font=\tiny] \arrow[urrr,bend left = 15,yshift=-1ex,"\lambda"' font=\tiny] & \mbox{\fbox{\begin{tabular}{c} Partially \\ Immune \\($S_1$)\end{tabular}}}  \arrow[l,"k\alpha"' font=\tiny] \arrow[drr,bend right = 10,"\beta_2"' font=\tiny] \arrow[urr,bend left = 10, "\lambda"' font=\tiny] & \cdots \arrow[l,"k\alpha"' font=\tiny] & \mbox{\fbox{\begin{tabular}{c} Partially \\ Immune \\($S_r$)\end{tabular}}}  \arrow[l,"k\alpha"' font=\tiny] \arrow[d, "\beta_2" font=\tiny] \arrow[u,"\lambda"' font=\tiny] & \cdots \arrow[l,"k\alpha"' font=\tiny] & \mbox{\fbox{\begin{tabular}{c} Partially \\ Immune \\($S_{k-1}$)\end{tabular}}} \arrow[l,"k\alpha"' font=\tiny] & \mbox{\fbox{\begin{tabular}{c} Partially \\ Immune \\($S_k$)\end{tabular}}} \arrow[l,"k\alpha"' font=\tiny] \\
 & & & \mbox{\fbox{\begin{tabular}{c} Strain 2 \\($I_2$)\end{tabular}}} \arrow[urrr,bend right=15,"\gamma"' font=\tiny] & & & \\
 & & & \mbox{\fbox{\begin{tabular}{c} Strain 1 \\($I_1$)\end{tabular}}} \arrow[uurrr,bend right = 25,"\gamma"' font=\tiny] & & &
\end{tikzcd}
\end{equation}
The model \eqref{SIR6} can be corresponded to the following system of ordinary differential equations:
\begin{equation}
    \left\{ \; \; \;
    \label{SIR6-DE}
    \begin{aligned}
        \frac{dS_0}{dt} & = -\frac{\beta_1}{N} S_0 I_1 -\frac{\beta_2}{N} S_0 I_2 + k\alpha S_1 - \lambda S_0 + \alpha V&&  \\
        \frac{dS_i}{dt} & =  -\frac{\beta_2}{N} S_i I_2 + k\alpha {S}_{i+1} - k\alpha S_i - \lambda S_i, \; \; \; \; \; \; \; \; \; && i = 1, \ldots, r \\
        \frac{dS_i}{dt} & = k\alpha {S}_{i+1} - k\alpha S_i, && i = r+1, \ldots, k-1 \\
        \frac{dS_k}{dt} & = \gamma (I_1+ I_2) - k\alpha S_k && \\
        \frac{dI_1}{dt} & = \frac{\beta_1}{N} S_0 I_1 - \gamma I_1&& \\
        \frac{dI_2}{dt} & = \frac{\beta_2}{N} I_2 \sum_{i=0}^r S_i - \gamma I_2&&\\
        \frac{dV}{dt} & = \lambda \sum_{i=0}^r S_i  - \alpha V.
    \end{aligned}
    \right.
\end{equation}

In Appendix \ref{sec:32} we prove the following.
\begin{theorem}
    \label{thm:4}
    The system of distributed delay differential equations \eqref{SIR5-DE} is equivalent to the system of ordinary differential equations \eqref{SIR6-DE} with the following variable substitutions:
\[
    \begin{aligned}
        S_0(t) & = S(t) \\
        S_i(t) & = \frac{1}{k\alpha} \int_{-\infty}^t \Gamma(\tau) e^{-\frac{\beta_2}{N} \int_{\tau}^t I_2(s) \; ds} e^{-\lambda(t-\tau)} g_{k\alpha}^{r-i+1}(t-\tau) \; d\tau & & i = 1, \ldots, r\\
        S_i(t)  & = \frac{1}{k\alpha} \int_{-\infty}^t \gamma(I_1(\tau)+I_2(\tau)) g_{k\alpha}^{k-i+1}(t-\tau)\; d\tau & & i = r+1, \ldots, k\\
    \end{aligned}
\]
where $R(t) = \sum_{i=1}^r S_i(t)$.
\end{theorem}

In Appendix \ref{sec:22} we compute the steady states of \eqref{SIR6-DE} and use those to derive the following reproduction numbers:
\begin{equation}
    \label{R06}
    \begin{aligned}
    \mathscr{R}_1 & = \frac{\beta_1}{\gamma} \left( \frac{\alpha}{\alpha+\lambda} \right), \mathscr{R}_2 = \frac{\beta_2}{\gamma} \left( \frac{\alpha}{\alpha+\lambda} \right), \mathscr{R}_0  = \max \{ \mathscr{R}_1, \mathscr{R}_2 \},\\
    \mathscr{R}_{12} & = \frac{\beta_1}{\beta_2} \left( \frac{
            (k\alpha (\beta_2 - \gamma) - \gamma \lambda k)
            \left( \frac{
                k\alpha ((k-r) \gamma + k\alpha)
            } {
                (k\alpha)^2
              + k\alpha ((k-r-1) \gamma + \beta_2 + \lambda)
              - \gamma \lambda r
            } \right)^r
          + \lambda ((k-r) \gamma + k\alpha)
        } {
            k\alpha (\beta_2 - \gamma + \lambda) - \gamma \lambda r
        } \right),\\
    \mathscr{R}_{21} & = \frac{\beta_2}{\beta_1} \left( \frac{
             \alpha \beta_1 k(k \alpha)^r
          + (\gamma \lambda r -((\alpha + \gamma) \lambda + \alpha \beta_1) k)
            (\alpha k + \lambda)^r
        } {
            k \gamma(\alpha + \lambda)(k \alpha)^r
          + (\gamma \lambda r-((\alpha + 2 \gamma) \lambda + \alpha \gamma) k )
            (\alpha k + \lambda)^r
        } \right).
    \end{aligned}
\end{equation}

\subsection{Model Analysis Techniques}

Our analysis of the models \eqref{SIR4} and \eqref{SIR6} focus on two features: bifurcation diagrams and steady state diagrams. Both analysis techniques focus on understanding how the competitive balance of two strains changes with respect to changes in a population's effective immunity level.

\paragraph{Population Vaccination Behavior}

The vaccination rate parameter, $\lambda$, is challenging to interpret directly. Instead, we update the parameter $0 \leq \epsilon \leq 1$ introduced in Section \ref{sec:example}. Since this value is attained from the disease-free steady state, it will roughly scale with a population's effective vaccination level (i.e. the proportion of a population who are up-to-date on their vaccinations and for whom the vaccine is effective).

For the integrated immunity model \eqref{SIR4}, we let $\epsilon$ correspond to the proportion of the population at steady state who is immune to both Strain 1 and Strain 2 (i.e. those who are in states $S_{r+1}$ through $S_k$). We therefore consider the equation
\begin{equation}
\label{111}
\epsilon = \frac{1}{N}\sum_{i=r+1}^k S_i^* = \frac{1}{N} \sum_{i=r+1}^k \frac{\lambda N}{k \alpha + (k-r) \lambda} = \frac{(k-r)\lambda}{k \alpha + (k-r)\lambda}
\end{equation}
where the values for $S_i^*$ at the disease-free steady state for \eqref{SIR4-DE} can be found in Appendix \ref{sec:21}. In order to re-parametrize \eqref{SIR4} by $\epsilon$ rather than $\lambda$, we solve \eqref{111} for $\lambda$ to get
\[\lambda = \frac{\epsilon \alpha k}{(1-\epsilon)(k-r)}.\]
For the separated immunity model \eqref{SIR6}, we let $\epsilon$ correspond to the proportion of the population at steady state who is vaccinated (i.e. those who are in state $V$). We therefore consider the equation
\[\epsilon = \frac{1}{N}V^* = \frac{1}{N}\left( \frac{\lambda N}{\alpha + \lambda} \right)= \frac{\lambda }{\alpha + \lambda}\]
which can be solved in terms of $\lambda$ to get
\[\lambda = \frac{\alpha \epsilon}{1 - \epsilon}.\]

\paragraph{Bifurcation Diagrams}

For our bifurcation diagrams, we will focus on the population's effective vaccination level $\epsilon$ and the factors which give each strain its competitive advantage. Since Strain 1 has the advantage of higher transmissibility ($\beta_1 > \beta_2$), the degree of advantage is parametrized by $\beta_1$. Since Strain 2 has the advantage of higher immune resistance, the degree of advantage is parametrized by $r$. We compute bifurcation diagrams in the four qualitatively distinct behaviors in Definition \ref{behavior} for $\epsilon$ versus either $\beta_1$ or $r$. The boundaries are determined by the equations $\mathscr{R}_1 = 1$, $\mathscr{R}_2 = 1$, $\mathscr{R}_{12} = 1$, and $\mathscr{R}_{21} = 1$ determined from \eqref{R05} and \eqref{R06}.

\paragraph{Steady State Diagrams}

In order to represent how a system's long-term behavior depends on the effective vaccination rate $\epsilon$, we plot the stable steady state values where the total infection ($I_1+I_2$, upper curve) is broken down by Strain 1 ($I_1$, shaded red) and Strain 2 ($I_2$, shaded blue) as a function of the effective vaccination level of the population ($\epsilon$). The steady state values of $I_1$, $I_2$, and $I_1+I_2$ are determined numerically by simulating to $t=10000$ the corresponding system with $N=1000$ from the initial condition $I_1(0) = I_2(0) = 10$, $S_0(0) = 980$, $S_i(0) = 0$ for $i=1, \ldots, k$, and $V(0)=0$ (if necessary). We leave the task of establishing explicit results on the parameter regions required for stability of the various steady states as future work.



\section{Discussion}
\label{sec:discussion}

In this section, we investigate how the competitive fitness of strains is impacted by changes in a population's effective vaccination level. 

\subsection{Comparison of Integrated and Separated Immunity}

In Figure \ref{fig:2}, we consider the effect of changing the population's effective immunity ($\epsilon$), the transmissibility of Strain 1 ($\beta_1$), and the degree of immune resistance of Strain 2 ($r$) on the behavior of the integrated and separated immunity models \eqref{SIR4} and \eqref{SIR6}, respectively. In the bifurcation diagrams, we follow the qualitatively distinct regions outlined in Definition \ref{behavior}: DF - Disease-free behavior; S1 - Strain 1-only behavior; S2 - Strain 2-only behavior; and C - Co-existence behavior.

\begin{figure}[t!]
\centering
     \begin{subfigure}[h]{0.3\textwidth}
         \centering
         \includegraphics[width=\textwidth]{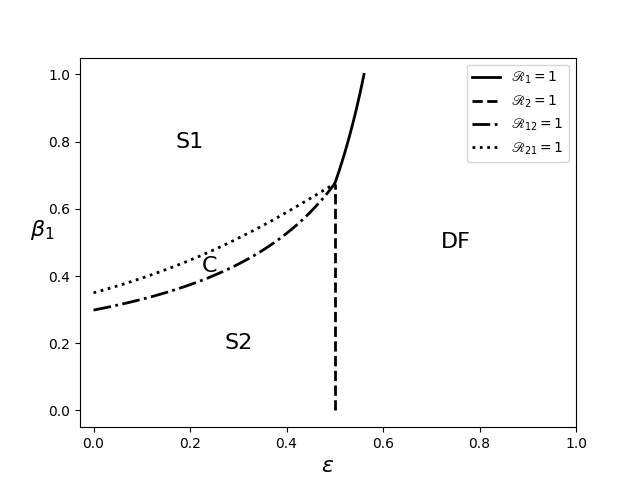}
         \caption{\footnotesize Integrated immunity model \eqref{SIR4-DE} with $0 \leq \beta_1 \leq 1$, $r=3$, and $0 \leq \epsilon \leq 1$.}
     \end{subfigure}
     \hfill
     \begin{subfigure}[h]{0.3\textwidth}
         \centering
         \includegraphics[width=\textwidth]{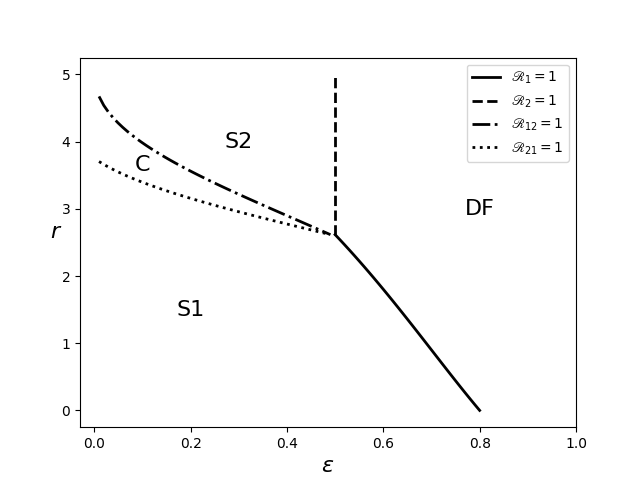}
         \caption{\footnotesize Integrated immunity model \eqref{SIR4-DE} with $\beta_1 = 0.5$, $0 \leq r \leq 5$, and $0 \leq \epsilon \leq 1$.}
     \end{subfigure}
     \hfill
     \begin{subfigure}[h]{0.3\textwidth}
         \centering
         \includegraphics[width=\textwidth]{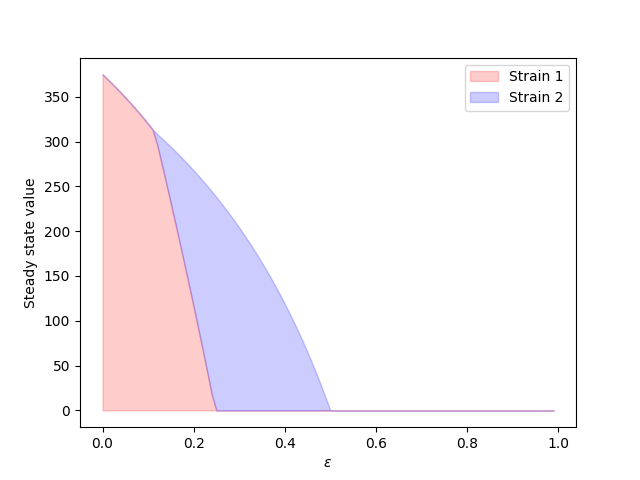}
         \caption{\footnotesize Integrated immunity model \eqref{SIR4-DE} with $\beta_1 = 0.4$, $r=3$, and $0 \leq \epsilon \leq 1$.}
     \end{subfigure}
     \newline
     \begin{subfigure}[h]{0.3\textwidth}
         \centering
         \includegraphics[width=\textwidth]{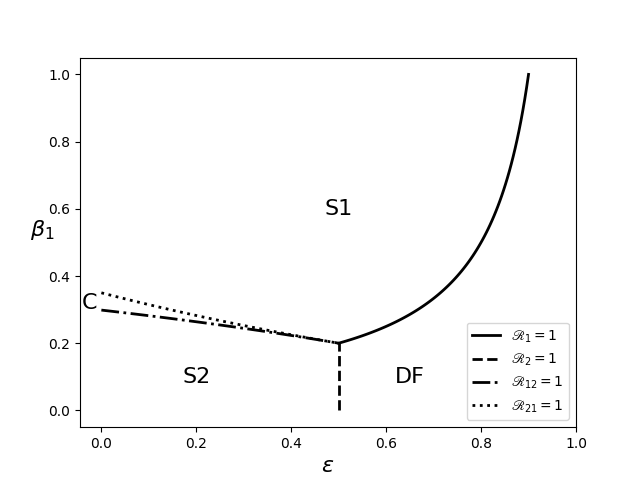}
         \caption{\footnotesize Separated immunity model \eqref{SIR6-DE} with $0 \leq \beta_1 \leq 1$, $r=3$, and $0 \leq \epsilon \leq 1$.}
     \end{subfigure}
     \hfill
     \begin{subfigure}[h]{0.3\textwidth}
         \centering
         \includegraphics[width=\textwidth]{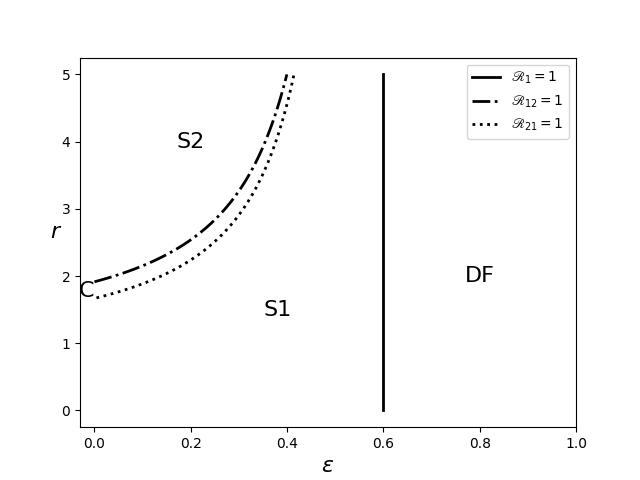}
         \caption{\footnotesize Separated immunity model \eqref{SIR6-DE} with $\beta_1 = 0.3$, $0 \leq r \leq 5$, and $0 \leq \epsilon \leq 1$.}
     \end{subfigure}
     \hfill
     \begin{subfigure}[h]{0.3\textwidth}
         \centering
         \includegraphics[width=\textwidth]{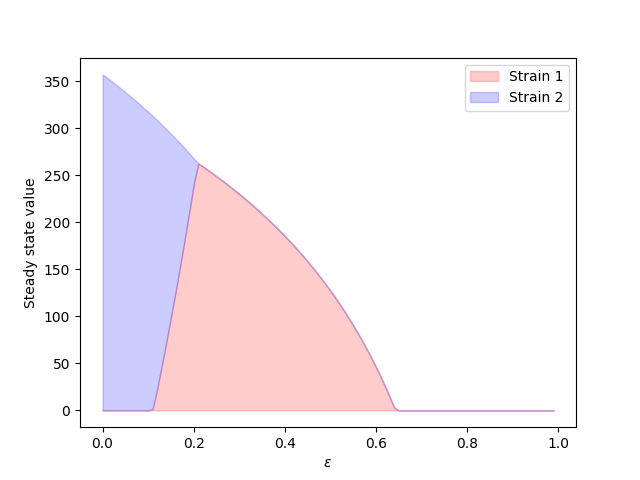}
         \caption{\footnotesize Separated immunity model \eqref{SIR6-DE} with $\beta_1 = 0.28$, $r=3$, and $0 \leq \epsilon \leq 1$.}
     \end{subfigure}
\caption{\footnotesize Bifurcation diagrams and steady states plots for the integrate and separated immunity models \eqref{SIR4} and \eqref{SIR6}, respectively, with $k=5$, $\beta_2=0.2$, $\gamma=0.1$, and $\alpha = 0.1$ (other parameters are indicated above). In the bifurcation diagrams (a), (b), (d), and (e), the regions as defined by Definition \ref{behavior} are labeled Disease-free behavior (\textbf{DF}), Strain 1-only behavior (\textbf{S1}), Strain 2-only behavior (\textbf{S2}), and Co-existence behavior (\textbf{C}). Only the relevant portion of the curves $\mathscr{R}_1 = 1$, $\mathscr{R}_2 = 1$, $\mathscr{R}_{12} = 1$, and $\mathscr{R}_{21}=1$ where the system transitions from one type of behavior to another are shown. Notice that the vertical bifurcation variable is either $\beta_1$ or $r$ which are the variables most closely related to the competitive advantage of Strain 1 and Strain 2, respectively. Specifically, increasing $\beta_1$ increases Strain 1's transmissibility advantage over Strain 2, and increasing $r$ increases Strain 2's immune resistance advantage over Strain 1. In all four bifurcation diagrams the horizontal bifurcation variable is the effective immunity level at the disease-free steady state ($\epsilon$). In (c) and (f), we present plots of the stable steady state values where the total infection ($I_1+I_2$, upper curve) is broken down by Strain 1 ($I_1$, shaded red) and Strain 2 ($I_2$, shaded blue) as a function of the effective vaccination level of the population ($\epsilon$). We note that systems have the capacity to exhibit all four qualitative distinct steady state behaviors outlined in Definition \ref{behavior} as $\epsilon$ is changed. Specifically, for the integrated immunity model we have: (a) Strain-1 only behavior for $0 \leq \epsilon < 0.113$; (b) Co-existence behavior for $0.113 < \epsilon < 0.247$; (c) Strain-2 only behavior for $0.247 < \epsilon < 0.5$; and (d) disease free behavior for $0.5 < \epsilon \leq 1$. For the separated immunity system the order of strain preference is reversed. Specifically, we observe: (a) Strain-2 only behavior for $0 \leq \epsilon < 0.110$; (b) Co-existence behavior for $0.110 < \epsilon < 0.208$; (c) Strain-1 only behavior for $0.208 < \epsilon < 0.643$; and (d) disease free behavior for $0.643 < \epsilon \leq 1$.}
\label{fig:2}
\end{figure}

We make the following observations:
\begin{enumerate}
\item
\textbf{Effect of increasing $\beta_1$:} In Figures \ref{fig:2}(a) and \ref{fig:2}(d) we see that increasing $\beta_1$ transitions both the integrated and separated immunity systems from Strain 2-only behavior (S2) to co-existence behavior (C) to Strain 1-only behavior (S1), or from disease free behavior (DF) to Strain 1-only behavior (S1). Consequently, increasing the transmissibility of the more transmissible strain increases the likelihood of that strain surviving and the more immune resistant strain being eliminated from the population.
\item 
\textbf{Effect of increasing $r$:} In Figures \ref{fig:2}(b) and \ref{fig:2}(e), we see that increasing $r$ transitions both the integrated and separated immunity systems from Strain 1-only behavior (S1) to co-existence behavior (C) to Strain 2-only behavior (S2), or from Strain 1-only behavior (S1) to disease free behavior (DF). Consequently, increasing the degree of immune resistance in the more immune resistant strain increases the likelihood of that strain surviving and eliminating the more transmissible strain from the population.
\item
\textbf{Effect of increasing $\epsilon$ (integrated immunity model \eqref{SIR4}):} In Figures \ref{fig:2}(a) and \ref{fig:2}(b), we see that as $\epsilon$ is increased (i.e. a higher proportion of the population vaccinates) the integrated immunity system ultimately transitions to disease-free behavior. However, we notice that there are certain ranges of $\beta_1$ and $r$ where increasing $\epsilon$ will result in all four qualitatively distinct behaviors being observed, in this order as $\epsilon$ increases: Strain 1-only behavior (S1), to co-existence behavior (C), to Strain 2-only behavior (S2), to disease-free behavior (DF). An illustration of this four-state behavior is given in Figure \ref{fig:2}(c). This suggests that, if vaccine and natural immunity are integrated, vaccination ultimately leads to the elimination of the disease but may also produce a competitive advantage for a more immune resistant strain.
\item
\textbf{Effect of increasing $\epsilon$ (separated immunity model \eqref{SIR6}):} In Figures \ref{fig:2}(d) and \ref{fig:2}(e), we see that as $\epsilon$ is increased the system ultimately transitions to disease-free behavior. We again observe ranges of the parameter $\beta_1$ and $r$ where all four behaviors are observed; however, the order is reversed from the integrated immunity model \eqref{SIR4}: Strain 2-only behavior (S2) to co-existence behavior (C) to Strain-1 only behavior (S1) to disease-free behavior (DF). An illustration of this four-state behavior is given in Figure \ref{fig:2}(f). This suggests that, if vaccine and natural immunity are separated, vaccination ultimately leads to the elimination of the disease but may also produce a competitive advantage for the more transmissible strain. 
\end{enumerate}

\subsection{Complex Behavior}


In general, we observe that for a given model, increasing the effective immunity level of the population ($\epsilon$) favors one strain over the other. Several examples, however, complicate this interpretation and hint at the complexity underlying the strain-over-strain interactions. We observe the following:

\begin{enumerate}
\item
\textbf{Parameter range dependence on which strain benefits from vaccination.} In general we expect that as the effective vaccination level of the population is increased, which strain outcompetes the other will be specific to the model studied rather than the choice of parameter values. However, consider the integrated immunity model \eqref{SIR4} with the parameter values $k=5$, $r=3$, $\beta_2 = 0.2$, $\gamma=0.1$, $\alpha=0.04$, and $N=1000$ (see Figure \ref{fig:3}(a)). We observe that, if $\beta_1 = 0.45$, then Strain 2 outcompetes Strain 1 as $\epsilon$ is increased (see Figure \ref{fig:3}(b)), but that if $\beta_1 = 0.7$, Strain 1 outcomes Strain 2 $\epsilon$ is increased (see Figure \ref{fig:3}(c)). Consequently, which strain is conferred an advantage as the population's effective vaccination level is raised depends upon the specific parameter values as well as which model is chosen. The phenomenon of parameter range dependence on which strain outcompetes the other was not observed in the separated immunity model \eqref{SIR6} but this cannot be ruled out.
\item 
\textbf{Non-monotone influence of vaccination on strain-over-strain fitness.} In general, we expect that, as the population's effective vaccination level $\epsilon$ rises, the advantage gained will be monotone in one strain or the other. Specifically, we do not expect that increasing the vaccination level to a certain value gives an advantage to one strain, and then increasing it further would start to give the advantage to the other. However, consider the integrated vaccination immunity model \eqref{SIR4} with the parameter values $k=4$, $r=3$, $\beta_2 = 0.169$, $\gamma = 0.1$, $\alpha=0.04$, and $N=1000$. This produces bifurcation diagram in $\epsilon$ vs. $\beta_1$ given in Figure \ref{fig:3}(d) and the steady state plots given in Figure \ref{fig:3}(e) ($\beta_1 = 0.75$) and Figure \ref{fig:3}(f) ($\beta_1 = 0.9$). We can see that the steady state value of Strain 2 is not monotone with respect to increases in $\epsilon$; rather, it increases and then decreases as $\epsilon$ rises. For the value $\beta_1 = 0.9$, the steady state value rises from below zero to a positive value, and then back below zero, resulting in a transition from Strain 1-only behavior, to co-existence behavior, back to Strain 1-only behavior. Although the proportion of total infection owing to Strain 2 is very low in this co-existence interval, it is still notable that the strain is able to survive only within this restricted range of $\epsilon$.
\end{enumerate}

\begin{figure}[t!]
\centering
     \begin{subfigure}[h]{0.3\textwidth}
         \centering
         \includegraphics[width=\textwidth]{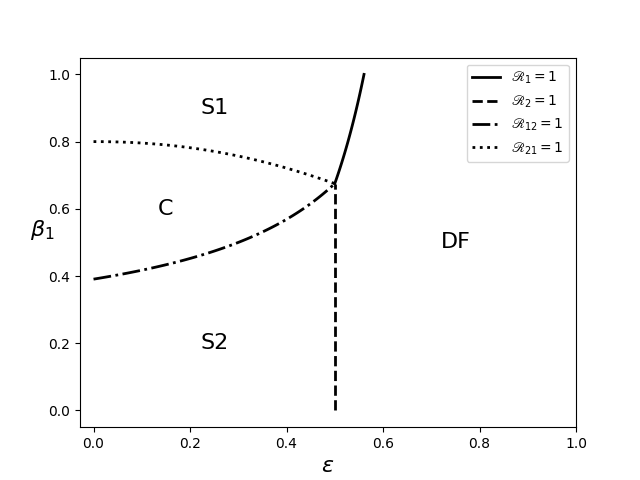}
         \caption{\footnotesize Bifurcation diagram}
     \end{subfigure}
     \hfill
     \begin{subfigure}[h]{0.3\textwidth}
         \centering
         \includegraphics[width=\textwidth]{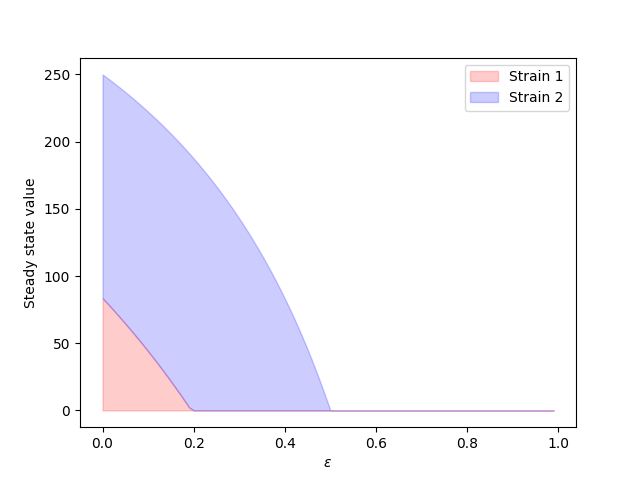}
         \caption{\footnotesize Steady state plot for $\beta_1=0.45$}
     \end{subfigure}
     \hfill
     \begin{subfigure}[h]{0.3\textwidth}
         \centering
         \includegraphics[width=\textwidth]{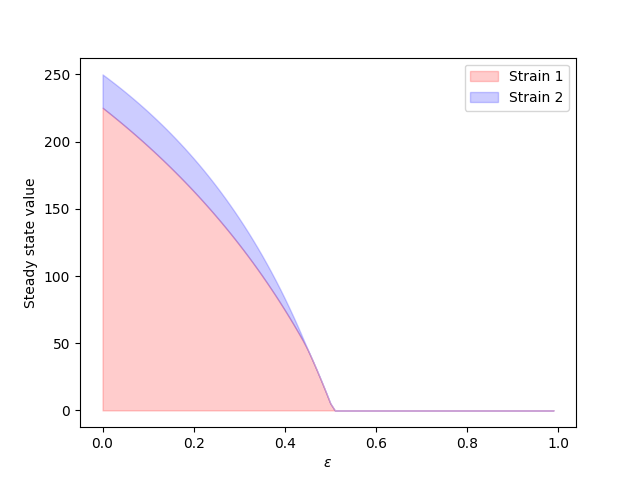}
         \caption{\footnotesize Steady state plot for $\beta_1=0.7$}
     \end{subfigure}
     \newline
          \begin{subfigure}[h]{0.3\textwidth}
         \centering
         \includegraphics[width=\textwidth]{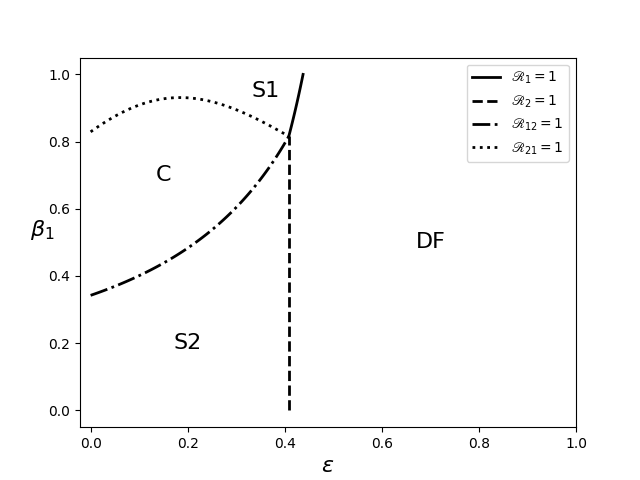}
         \caption{\footnotesize Bifucation diagram}
     \end{subfigure}
     \hfill
     \begin{subfigure}[h]{0.3\textwidth}
         \centering
         \includegraphics[width=\textwidth]{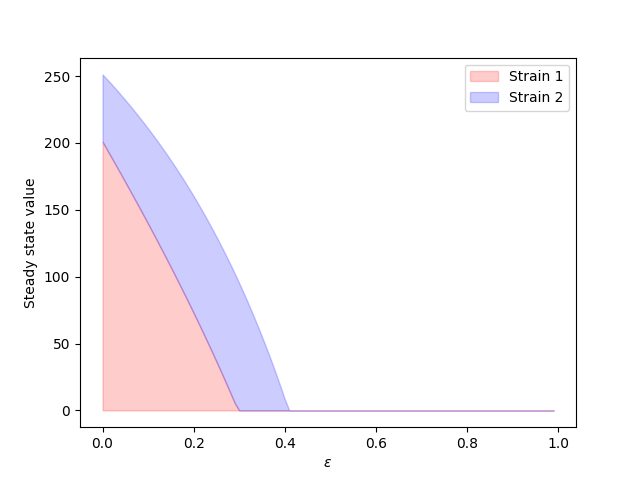}
         \caption{\footnotesize Steady state plot for $\beta_1=0.6$}
     \end{subfigure}
     \hfill
     \begin{subfigure}[h]{0.3\textwidth}
         \centering
         \includegraphics[width=\textwidth]{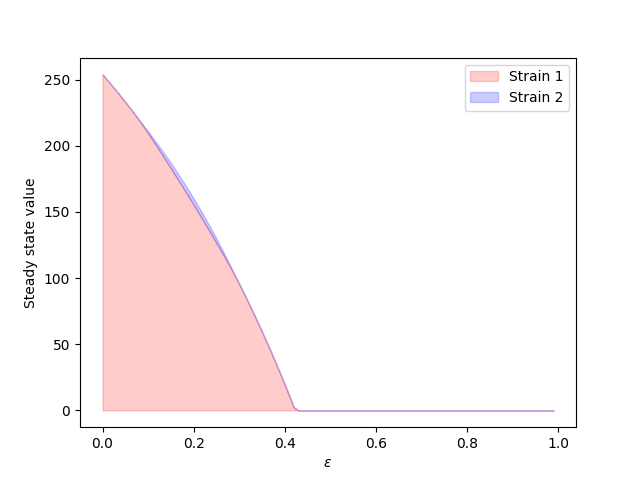}
         \caption{\footnotesize Steady state plot for $\beta_1=0.9$}
         \label{fig:}
     \end{subfigure}
\caption{\footnotesize Bifurcation diagrams and steady states plots for the integrated immunity model \eqref{SIR4} illustrating the complexity of possible behaviors. In the bifurcation diagrams, the regions as defined by Definition \ref{behavior} are labeled Disease-free behavior (\textbf{DF}), Strain 1-only behavior (\textbf{S1}), Strain 2-only behavior (\textbf{S2}), and Co-existence behavior (\textbf{C}). In Figure(a), a bifurcation plot of model \eqref{SIR4-DE} is given for the parameter values $r=3$, $k=5$, $\beta_2 = 0.2$, $\gamma=0.1$, $\alpha=0.04$, and $N=1000$. In (b) and (c), we plot the stable steady state values where the total infection ($I_1+I_2$, upper curve) is broken down by Strain 1 ($I_1$, shaded red) and Strain 2 ($I_2$, shaded blue) for the values $\beta_1 = 0.45$ and $\beta_1 = 0.7$. 
Notice that, in the region supporting co-existence behavior at $\epsilon = 0$ (roughly $0.4 \leq \beta_1 \leq 0.8$), which strain is conferred a competitive advantage as vaccination is increased depends on the value of $\beta_1$. Specifically, in (b) Strain 2 is initially favored and drives out Strain 1 as vaccination is increased while in (c) Strain 1 is initial favored and drives out Strain 2 as vaccination is increased. Notice that the total infection level responds similarly to an increased population immunity level; only the breakdown of the total infection by strains which changes. In (d), we present a bifurcation diagram of the integrated immunity model \eqref{SIR4} with the parameter values $k=4$, $r=3$, $\beta_2 = 0.169$, $\gamma = 0.1$, $\alpha=0.04$, and $N=1000$. In (e) and (f), we present steady state plots for the values $\beta_1 = 0.6$ and $\beta_2 = 0.9$. Note that the steady state value of Strain 2 is not monotone as $\epsilon$ is increased. In particular, for $\beta_2 = 0.9$, the steady state value can rise above zero and then below it as $\epsilon$ is increased, leading to a transition from Strain 1 only behavior, to co-existence behavior, back to Strain 1 only behavior.}
\label{fig:3}
\end{figure}

\subsection{Implications for Public Policy}

The analysis presented has important consequences in how infectious diseases stratified into variants are understood, modeled, and managed.

\begin{enumerate}
    \item \textbf{Vaccination impacts strain-over-strain fitness.} Given an effective and safe vaccine, it seems straight-forward that society should take measures to increase the population's effective vaccination level as that will decrease the prevalence of the disease in the population and, in the best case scenario, eliminate it. Our analysis supports this overall conclusion; however, it also highlights that the manner in which the disease's prevalence is reduced is not necessarily uniform across strains. In particular, it is possible that, as the population's effective vaccine immunity level rises, one strain may see an \emph{increase} in prevalence even as the general disease prevalence declines. This supports the hypothesis that immune-resistant variants of COVID-19, such as Omicron, have arisen due to the selective pressure of vaccination \cite{Kudriavtsev2022}. In such a case, an increase in strain prevalence does not indicate a failure on the part of the vaccination program or that the vaccine is not effective against the newly prevalent strain. Rather, the increase in prevalence results from the shift in selective pressure between the strains caused by vaccination.
    \item \textbf{The relationship between vaccine and natural immunity matters.} Given an effective and safe vaccine, it is tempting to disregard the mechanisms by which vaccine immunity and natural immunity work within the body. Our analysis, however, shows that raising the population's effective immunity level can affect the competition landscape between strains differently depending on how vaccine immunity and natural immunity are related. Specifically, if these immunities are integrated then an increase in a population's effective immunity level with tend to favor the immune-resistant strain; however, if these immunities are modeled separately, it tends to favor the more transmissible strain. Consequently, understanding the relationship between vaccine and natural immunity is essential to anticipating and understanding the population-level consequences of our public policies on vaccination.
    \item \textbf{Strain-specific consequences matter.} The analysis presented in this paper supports the implementation of effective and safe vaccines in containing infectious disease. Given that a population's effective vaccination level influences the strain-over-strain competitive fitness, however, it is conceivable that situations could arise where vaccination may not produce the anticipated benefits and the consequence of strain-specific behavior should be considered. For example, consider the following scenario: (1) vaccine and natural immunity are integrated; (2) it is impossible or impractical to raise the population's effective immunity to the point required for eradication; and (3) the more immune-resistant strain is significantly more virulent than the more transmissible strain. In such a case, it is possible that as a population's effective immunity level increases, the incidence of severe outcomes (e.g. hospitalizations, deaths, etc.) may rise as a higher proportion of infections are due to the more immune-resistant and more virulent strain. Note that this could occur even though the vaccine is effective in reducing the incidence of disease overall. A similar counter-intuitive situation could occur if the immunity due to vaccination is separated from that of previous infection and the more transmissible strain is also more virulent. These results suggest that attention should be paid to the understanding of detailed virus dynamics by policy makers.
\end{enumerate}

\section{Future Work}

We note that there are several limitations of our study. We propose the following as potential avenues for future work.
\begin{enumerate}
    \item {
        We have assumed that vaccine and natural immunity wane at the same rate. Several studies, however, have concluded that natural immunity provides longer protection from re-infection than vaccine immunity \cite{Gazit2021,Shenai2021,Townsend2022}. This assumption could be incorporated in the integrated models \eqref{SIR3} and \eqref{SIR4} by having the vaccination edges lead to a partial immune state between $S_r$ and $S_k$, and in the separated models \eqref{SIR5} and \eqref{SIR6} by having two $\alpha$ parameters for the two distinct notions of immunity.
    }
    \item {
        We have assumed that individuals may choose to vaccinate when they assess that they are first susceptible to the more immune-resistant Strain 2. We have also assumed that vaccination gives individuals the same amount of immune resistance, regardless of their pre-vaccine immunity level. Different decision criteria could be integrated into the model by allowing the outgoing edges for vaccination to be drawn from any ordered subsets of states drawn from $\{ S_0, \ldots, S_k \}$. Similarly, different vaccination effects could be incorporated by allowing the vaccination edges to direct to different states. For example, the assumption that vaccination provides a fixed bump in immunity could be incorporated by having each vaccination edge direct to the partial immunity state a fixed number of states ahead.
    }
    \item {
        We have assumed that vaccine and natural immunity are comparable. Some recent research, however, has suggested that natural immunity provides greater protection against COVID-19 variants than vaccine immunity, theorized to be due to the increased production of memory B-cells and cross-reaction T-cells in the natural immunity response \cite{Kudriavtsev2022}. This potential disparity in immune responses can be incorporated by adjusting the parameters and integration of the partial immunity chains for vaccine and natural immunity.
    }
     \item {
        Studies have suggested that hybrid immunity, resulting from both vaccination and previous infection, is superior to either vaccine or natural immunity by themselves \cite{Goldberg2022}. Our models do not consider this manner of integration of vaccine and natural immunity. Hybrid immunity can be incorporated into our models by considering individual immunity histories.
     }

\end{enumerate}

\paragraph{Acknowledgments}

MDJ, JP, and DAR are supported by National Science Foundation Grant No. DMS-2213390. BP is supported by National Science Foundation Grant No. DMS-2316809.

\begin{appendices}

\section{Reproduction Numbers}
\label{sec:2}

For computing the reproduction numbers of individual strains of an infectious disease, we follow the notation and methods of \cite{VANDENDRIESSCHE2002}.

We first re-index the state vector $\mathbf{x} \in \mathbb{R}_{\geq 0}^n$ so that $i= 1, \ldots, m$, $m \leq n$, corresponds to the infected compartments. For each infectious compartment $i=1, \ldots, m$, we let $\mathscr{F}_i$ denote the rate of new infections, $\mathscr{V}^+_i$ denote the rate of inflow into the compartment by any other means, and $\mathscr{V}^-_i$ denote the rate of outflow from each compartment by any other means. We can therefore write
\[\frac{dx_i}{dt} = \mathscr{F}_i - \mathscr{V}_i, \; i = 1, \ldots, m\]
where $\mathscr{V}_i = \mathscr{V}^-_i - \mathscr{V}^+_i$. Let
\[F = \left[ \frac{d\mathscr{F}_i}{dx_j} \right], \; i, j = 1, \ldots, m \]
and
\[V = \left[ \frac{d\mathscr{V}_i}{dx_j} \right], \; i, j = 1, \ldots, m \]
denote the Jacobians of $\mathscr{F} = [\mathscr{F}_1, \ldots, \mathscr{F}_m ]$ and $\mathscr{V} = [\mathscr{V}_1, \ldots, \mathscr{V}_m ]$, respectively, restricted to only infectious compartments. The next generation matrix is given by $FV^{-1}$. The basic reproduction number of the infectious disease is given by the spectral radius of the next generation matrix and evaluated the disease free steady state $\mathbf{x}_0$, i.e. $\mathscr{R}_0 = \rho(FV^{-1}(\mathbf{x}_0))$. Note that our notation is slightly modified from \cite{VANDENDRIESSCHE2002} since they define the matrices $F$ and $V$ to be those evaluated at the disease-free steady state $\mathbf{x}_0$ while we only evaluate at the steady state after computing the next-generation matrix $FV^{-1}$. This modification is made in order to give a common framework for strain-specific reproduction numbers, which requires evaluating at different steady states than the disease-free state.

The next generation method can be extended to derive strain-specific reproduction numbers in a natural way. We note that, for example, when considering infection by Strain $1$ specifically, uninfected and Strain $2$ infected states can both be considered uninfected. Consequently, when computing the basic reproduction number of Strain $1$, we may define the infectious compartments $i = 1, \ldots, m,$ to be only those resulting from infection with Strain $1$ and treat those infected with Strain $2$ as uninfected. The basic reproduction number of Strain 1, $\mathscr{R}_1$, is the spectral radius of the resulting next generation matrix evaluated at the disease-free steady state (i.e. $\mathscr{R}_1 = \rho(FV^{-1}(\mathbf{x}_0))$, while the basic reproduction number of Strain $1$ in the presence of Strain $2$, $\mathscr{R}_{12}$, is the spectral radius of this next generation matrix evaluated at the Strain 2 only steady state (i.e. $\mathscr{R}_{12} = \rho(FV^{-1}(\mathbf{x}_1))$). Collectively, we refer to $\mathscr{R}_1$ and $\mathscr{R}_{12}$ as the Strain 1-specific reproduction numbers and $\mathscr{R}_2$ and $\mathscr{R}_{21}$ as the Strain 2-specific reproduction numbers. 


We now apply these techniques to the models \eqref{SIR4} and \eqref{SIR6}. Notice that the rate equations for $I_1$ and $I_2$ are the same in the system of differential equations for the integrated mode \eqref{SIR4-DE} and the separated model \eqref{SIR6-DE}. Consequently, up to evaluating at the model-specific steady states, we have the same next generation matrix. For Strain 1, we have
\begin{equation}
    \label{x0-dfe-R1}
    \mathscr{F} = \frac{\beta_1}{N} S_0 I_1, \mathscr{V} = \gamma I_1, F = \frac{\beta_1}{N} S_0, V = \gamma, \mbox{ and } FV^{-1} = \frac{\beta_1 S_0}{\gamma N}
\end{equation}
where we have omitted matrix notation for simplicity since all of the matrices are $1 \times 1$. For Strain 2, we have
\begin{equation}
    \label{x0-dfe-R2}
    \mathscr{F} = \frac{\beta_2}{N} I_2 \sum_{i=0}^r S_i, \mathscr{V} = \gamma I_2, F = \frac{\beta_2}{N} \sum_{i=0}^r S_i, V = \gamma, \mbox{ and } FV^{-1} = \frac{\beta_2}{\gamma N} \sum_{i=0}^r S_i.
\end{equation}
To compute the reproduction numbers $\mathscr{R}_0$, $\mathscr{R}_1$, $\mathscr{R}_2$, $\mathscr{R}_{12}$, and $\mathscr{R}_{21}$, we evaluate the next generation matrix $FV^{-1}$ \eqref{x0-dfe-R1} or \eqref{x0-dfe-R2} at the relevant steady states. Specifically, for $\mathscr{R}_1$ we evaluate \eqref{x0-dfe-R1} at the disease-free state $\mathbf{x}_0$, for $\mathscr{R}_2$ we evaluate \eqref{x0-dfe-R2} at the disease-free state $\mathbf{x}_0$, for $\mathscr{R}_{12}$ we evaluate \eqref{x0-dfe-R1} at the Strain-2 only steady state $\mathbf{x}_2$, for $\mathscr{R}_{21}$ we evaluate \eqref{x0-dfe-R2} at the Strain-1 only steady state $\mathbf{x}_1$, and for $\mathscr{R}_0$ we take the max of $\mathscr{R}_1$ and $\mathscr{R}_2$.

\subsection{Integrated Immunity Model}
\label{sec:21}

\subsubsection{Strain-Specific Basic Reproduction Numbers ($\mathscr{R}_1$, $\mathscr{R}_2$)}
To determine the reproduction numbers for Strains 1 and Strain 2, we must first solve for the disease-free equilibrium $\mathbf{x}_0$.
By substituting $I_1 = I_2 = 0$ into \eqref{SIR4-DE}, and setting all derivatives to 0, we find that $\mathbf{x}_0$ satisfies the system
\begin{equation}
\label{SIR4-x0}
\begin{aligned}
    \begin{cases}
        0 = k\alpha S_1 - \lambda S_0,
    \\  0 = k\alpha S_{i+1} - (k\alpha + \lambda) S_i,& i = 1, \ldots, r
    \\  0 = k\alpha S_{i+1} - k\alpha S_i,            & i = r+1, \ldots, k-1
    \\  \displaystyle{0 = -k\alpha S_k + \lambda \sum_{i=0}^r S_i.}
    \end{cases}
\end{aligned}
\end{equation}
The first three equations in \eqref{SIR4-x0} can be recursively solved in terms of $S_0$ to get:
\[
\begin{aligned}
    \begin{cases}
      \displaystyle{S_i = \frac{\lambda}{k\alpha} \left(
                \frac{k\alpha + \lambda}{k\alpha}
            \right)^{i-1} S_0,}                      & i = 1, \ldots, r+1
    \\  \displaystyle{S_i = \frac{\lambda}{k\alpha} \left(
                \frac{k\alpha + \lambda}{k\alpha}
            \right)^r S_0,}                          & i = r+1, \ldots, k.
    \end{cases}
\end{aligned}
\]
We can then use the geometric series to deduce that
\begin{equation}
    \label{SIR4-x0-geometric}
    \sum_{i=0}^r S_i = S_0 + \sum_{i=1}^r \frac{\lambda}{k\alpha} \left(
                \frac{k\alpha + \lambda}{k\alpha} \right)^{i-1} S_0 = \left(
                \frac{k\alpha + \lambda}{k\alpha} \right)^{r} S_0.
\end{equation}
From here, the population conservation equation $N = \sum_{i=0}^k S_i$ gives
\begin{align*}
    N = \sum_{i=0}^r S_i + \sum_{i=r+1}^k S_i
        & = \left( \frac{k\alpha + \lambda}{k\alpha} \right)^{r} S_0
        +  \frac{(k-r)\lambda}{k\alpha} \left(
                \frac{k\alpha + \lambda}{k\alpha}
            \right)^r S_0
        \\ & = S_0 \left( \frac{
            k\alpha + (k - r)\lambda
        }{
            k\alpha
        } \right) \left( \frac{
            k\alpha + \lambda
        } {
            k\alpha
        } \right)^r
\end{align*}
from which we can solve for $S_0$ to get
\begin{align}
\label{SIR4-x0-S0}
    S_0 = \frac{
            N k\alpha
        } {
            k\alpha + (k - r)\lambda
        } \left( \frac{
            k\alpha
        } {
            k\alpha + \lambda
        } \right)^r.
\end{align}
Substituting the disease-free steady state \eqref{SIR4-x0-S0} into the next generation matrices \eqref{x0-dfe-R1} and \eqref{SIR4-x0-geometric} gives 
\begin{align*}
    \mathscr{R}_1
        = \frac{\beta_1}{\gamma} \left(
            \frac{k\alpha}{k\alpha + (k - r)\lambda}
        \right) \left(
            \frac{k\alpha}{k\alpha + \lambda}
        \right)^r \mbox{ and } \mathscr{R}_2
        = \frac{\beta_2}{\gamma} \left(\frac{k\alpha}{k\alpha + (k - r)\lambda}\right).
\end{align*}

\subsubsection{Strain 1 Competitive Reproduction Number ($\mathscr{R}_{12}$)}
Next, the strain 1 competitive reproduction number is found at the Strain 2-only equilibrium $\mathbf{x}_2$. We substitute $I_2 > I_1 = 0$ into \eqref{SIR4-DE} and set all derivatives to 0, which results in the system
\begin{equation} \label{x2} \begin{cases}
    0 = \displaystyle {
            -\frac{\beta_2}{N} S_0 I_2
          + k\alpha S_1
          - \lambda S_0
        }
\\  0 = \displaystyle {
            -\frac{\beta_2}{N} S_i I_2
          + k\alpha S_{i+1}
          - k\alpha S_i
          - \lambda S_i
        },                                & i = 1, \ldots, r
\\  0 = k\alpha ( S_{i+1} - S_i),    & i = r+1, \ldots, k-1
\\  0 = \displaystyle{
            \gamma I_2
          - k\alpha S_k
          + \lambda \sum_{i=0}^r S_i
        }
\\  0 = \displaystyle{
            \left( \frac{\beta_2}{N} \sum_{i=0}^r S_i
          - \gamma \right) I_2.
        }
\end{cases} \end{equation}
Once again, there is a recursion in the first two equations, which can be solved explicitly in terms of $S_0$ and $I_2$. The 5th equation of \eqref{x2} can be solved for $\sum_{i=0}^r S_i$ and then substituted into the 4th equation of \eqref{x2} solved for $S_k$. This transforms \eqref{x2} into
\begin{align}
\label{SIR4-x2-eqs}
    \begin{cases}
        S_i = \displaystyle{
                \left(
                    \frac
                        {\beta_2 I_2 + N (\lambda + k\alpha) }
                        {N k\alpha}
                \right)^{i-1} \left(
                    \frac
                        {\beta_2 I_2 + N \lambda}
                        {N k\alpha}
                \right) S_0
            },      & i = 1, \ldots, r+1
    \\  S_i = \displaystyle{
                \frac{\gamma\left( \beta_2 I_2
              +N\lambda \right)}{k\alpha \beta_2} 
            },      & i = r+1, \ldots, k
    \\  \displaystyle{
            \sum_{i=0}^r S_i = \frac{N\gamma}{\beta_2}
        }
    \end{cases}
\end{align}
We can substitute the last two equations of \eqref{SIR4-x2-eqs} into the population conservation equation to get
\begin{align*}
    N & = \sum_{i=0}^k S_i + I_2= \sum_{i=0}^r S_i + \sum_{i=r+1}^k S_i + I_2 = \frac{N\gamma}{\beta_2} + \frac{(k-r)\gamma\left( \beta_2 I_2
              +N\lambda \right)}{k\alpha \beta_2}  + I_2. 
\end{align*}
This can be solved for $I_2$ to show that Strain 2 is endemically present at
\begin{align}
\label{SIR4-x2-I2}
    I_2 = \frac{N}{\beta_2} \left( \frac{
            \beta_2 k\alpha - \gamma ((k - r)\lambda + k\alpha)
        } {
            (k - r)\gamma + k\alpha
        } \right)
\end{align}
Now, to compute $\mathscr{R}_{12}$ we only have to solve for $S_0$. We substitute \eqref{SIR4-x2-I2} into the first equation of \eqref{SIR4-x2-eqs} to get
\[
    S_i = \left(
            \frac{
                \beta_2 + \lambda + (k - r - 1)\gamma + k\alpha
            } {
                (k - r)\gamma + k\alpha
            }
        \right)^{i-1} \left(
            \frac{\beta_2 + \lambda - \gamma}{(k - r)\gamma + k\alpha}
        \right) S_0,
i = 1, \ldots, r+1
\]
which is a recurrence relation in terms of $S_0$ only. Note that the geometric series gives
{ \footnotesize
\begin{align*}
  & \sum_{i=1}^r \left(
        \frac{
            \beta_2 + \lambda + (k - r - 1)\gamma + k\alpha
        } {
            (k - r)\gamma + k\alpha
        }
    \right)^{i-1}
    = \left( \frac{
            (k - r)\gamma + k\alpha
        } {
            \beta_2 + \lambda - \gamma
        } \right) \left(
            \left( \frac{
                \beta_2 + \lambda - \gamma + (k - r)\gamma + k\alpha
            } {
                (k - r)\gamma + k\alpha
            } \right)^r
          - 1
        \right),
\end{align*}
}
so that
\begin{align*}
    \sum_{i=0}^r S_i
     & = S_0 + S_0 \left(
            \frac{\beta_2 + \lambda - \gamma}{(k - r)\gamma + k\alpha}
        \right) \sum_{i=1}^r \left(
            \frac{
                \beta_2 + \lambda + (k - r - 1)\gamma + k\alpha
            } {
                (k - r)\gamma + k\alpha
            }
        \right)^{i-1}
\\  & = S_0 \left( \frac{
            \beta_2 + \lambda - \gamma + (k - r)\gamma + k\alpha
        } {
            (k - r)\gamma + k\alpha
        } \right)^r.
\end{align*}
Since $\sum_{i=0}^r S_i = \displaystyle{\frac{N\gamma}{\beta_2}}$ by the third equation of \eqref{SIR4-x2-eqs}, we have
\begin{align*}
    S_0 = \frac{N\gamma}{\beta_2} \left( \frac{
            (k - r)\gamma + k\alpha
        } {
            \beta_2 + \lambda - \gamma + (k - r)\gamma + k\alpha
        } \right)^r.
\end{align*}
Substituting this value in \eqref{x0-dfe-R1} gives
\begin{align*}
    \mathscr{R}_{12}
        = \frac{\beta_1 S_0}{N\gamma}
        = \frac{\beta_1}{\beta_2} \left( \frac{
            (k - r)\gamma + k\alpha
        } {
            \beta_2 + \lambda - \gamma + (k - r)\gamma + k\alpha
        } \right)^r.
\end{align*}

\subsubsection{Strain 2 Competitive Reproduction Number ($\mathscr{R}_{21}$)}

To determine the strain 2 competitive reproduction numbers, we must first solve for the Strain-1 only equilibrium $\mathbf{x}_1$.
By substituting $I_1 > I_2 = 0$ into \eqref{SIR4-DE}, and setting all derivatives to 0, we find that $\mathbf{x}_1$ satisfies the system
\begin{align}
\label{SIR4-x1}
    \begin{cases}
        \displaystyle{0 = -\frac{\beta_1}{N} S_0 I_1 + k\alpha S_1 - \lambda S_0,}
    \\  0 = k\alpha S_{i+1} - (k\alpha + \lambda)S_i,             & i = 1, \ldots, r
    \\  0 = k\alpha (S_{i+1} - S_i),                              & i = r+1, \ldots, k-1
    \\  \displaystyle{0 = \gamma I_1 - k\alpha S_k + \lambda \sum_{i=0}^r S_i}
    \\  \displaystyle{0 = \left(\frac{\beta_1}{N} S_0 - \gamma \right) I_1.}
    \end{cases}
\end{align}
Since $I_1 > 0$, we can solve for $S_0$ in the 5th equation of \eqref{SIR4-x1} and then recursively derive equations for $S_i$, $i=1, \ldots, k$, in terms of $I_1$ using the remaining four equations of \eqref{SIR4-x1} to get the following:
\[
    \begin{cases}
        \displaystyle{S_0 = \frac{N\gamma}{\beta_1}}
    \\  \displaystyle{S_i = \frac{\gamma\left(
                \beta_1 I_1 + N \lambda
            \right)}{k\alpha\beta_1}  \left(
                \frac{k\alpha + \lambda}{k\alpha}
            \right)^{i-1},}                                & i = 1, \ldots, r
    \\  \displaystyle{S_i = \frac{\gamma\left(
                \beta_1 I_1 + N\lambda
            \right)}{k\alpha\beta_1}   \left(
                \frac{k\alpha + \lambda}{k\alpha}
            \right)^r,}                                    & i = r+1, \ldots, k.
    \end{cases}
\]
We can then use the geometric series to deduce that
\begin{align}
\label{SIR4-x1-geometric}
\sum_{i=1}^r S_i & =  \frac{\gamma\left(
                \beta_1 I_1 + N\lambda
            \right)}{k\alpha\beta_1} \sum_{i=1}^r  \left(
            \frac{k\alpha + \lambda}{k\alpha}
        \right)^{i-1}  = \frac{\gamma\left(
                \beta_1 I_1 + N\lambda
            \right)}{\lambda \beta_1} \left( \left(
            \frac{k\alpha + \lambda}{k\alpha}
        \right)^{r} - 1 \right).
\end{align}
The population conservation equation becomes
\begin{align*}
    N & = \sum_{i=0}^k S_i + I_1  = S_0 + \sum_{i=1}^{r} S_i + \sum_{i=r+1}^k S_i + I_1 \\
      & = \frac{N\gamma}{\beta_1} + \frac{\gamma\left(
            \beta_1 I_1 + N\lambda
        \right)}{\lambda \beta_1} \left( \left(
        \frac{k\alpha + \lambda}{k\alpha}
    \right)^{r} - 1 \right) + \frac{(k-r)\gamma\left(
            \beta_1 I_1 + N\lambda
        \right)}{k\alpha\beta_1}   \left(
            \frac{k\alpha + \lambda}{k\alpha}
        \right)^r + I_1
\end{align*}
which has the solution
\begin{align}
\label{SIR4-x1-I1}
    I_1 = \frac{N\lambda}{\beta_1} \left(
            \frac{
                (\beta_1 + \lambda - \gamma) (k\alpha)^{r+1}
            }{
                (\lambda - \gamma) (k\alpha)^{r+1}
              + \gamma (k\alpha + \lambda(k - r)) (k\alpha + \lambda)^r
            }
          - 1
        \right).
\end{align}
We then substitute \eqref{SIR4-x1-I1} into \eqref{SIR4-x1-geometric} and then \eqref{SIR4-x1-geometric} into $FV^{-1}$ in \eqref{x0-dfe-R2}. Simplifying, this gives
\begin{align*}
    \mathscr{R}_{21} = \frac{\beta_2}{\beta_1} \left(
        \frac{
            (\gamma\lambda (k - r) + (\beta_1 + \lambda) k\alpha) (k\alpha + \lambda)^r
          - \beta_1 (k\alpha)^{r+1}
        }{
            \gamma (k\alpha + (k - r)\lambda) (k\alpha + \lambda)^r
          + (\lambda - \gamma) (k\alpha)^{r+1}
        }
    \right).
\end{align*}

\subsection{Separated Immunity Model}
\label{sec:22}

\subsubsection{Strain-Specific Basic Reproduction Numbers ($\mathscr{R}_1$, $\mathscr{R}_2$)}

To find the disease free equilibrium $\mathbf{x}_0$, we substitute $I_1 = I_2 = 0$ into \eqref{SIR6-DE} to produce
\begin{equation}
    \left\{ \; \; \;
    \label{SIR6-DE-x0}
    \begin{aligned}
        0 & = k\alpha S_1 - \lambda S_0 + \alpha V&&  \\
        0 & = k\alpha {S}_{i+1} - k\alpha S_i - \lambda S_i, \; \; \; \; \; \; \; \; \; && i = 1, \ldots, r \\
        0 & = k\alpha {S}_{i+1} - k\alpha S_i, && i = r+1, \ldots, k-1 \\
        0 & = - k\alpha S_k && \\
        0 & = \lambda \sum_{i=0}^r S_i  - \alpha V.
    \end{aligned}
    \right.
\end{equation}
It follows immediately from the second through fourth equations of \eqref{SIR6-DE-x0} that $S_i = 0$ for $i=1, \ldots, k$. Consequently, from the first and fifth equation of \eqref{SIR6-DE-x0}, we have $V = \frac{\lambda}{\alpha} S_0$. 
The population conservation equation then gives
\begin{align*}
    N   = S_0 + V
        = S_0 + \frac{\lambda}{\alpha} S_0
        = \left( \frac{\lambda + \alpha}{\alpha} \right) S_0
\end{align*}
so that $S_0 = \frac{N\alpha}{\lambda + \alpha}$. 
Substituting this into the next generation matrices $FV^{-1}$ \eqref{x0-dfe-R1} and \eqref{x0-dfe-R2} gives
\begin{align*}
    \mathscr{R}_1 = \frac{\beta_1}{\gamma} \left( \frac{\alpha}{\lambda + \alpha}\right)     \mbox{ and }
    \mathscr{R}_2 = \frac{\beta_2}{\gamma} \left( \frac{\alpha}{\lambda + \alpha}\right).
\end{align*}

\subsubsection{Strain 1 Competitive Reproduction Number ($\mathscr{R}_{12}$)}
To find the Strain 2 only equilibrium, we substitute $I_2 > I_1 = 0$ into \eqref{SIR6-DE} to get 
\begin{equation}
    \left\{ \; \; \;
    \label{SIR6-DE-x2}
    \begin{aligned}
        0 & = -\frac{\beta_2}{N} S_0 I_2 + k\alpha S_1 - \lambda S_0 + \alpha V&&  \\
        0 & =  -\frac{\beta_2}{N} S_i I_2 + k\alpha {S}_{i+1} - k\alpha S_i - \lambda S_i, \; \; \; \; \; \; \; \; \; && i = 1, \ldots, r \\
        0 & = k\alpha {S}_{i+1} - k\alpha S_i, && i = r+1, \ldots, k-1 \\
        0 & = \gamma I_2 - k\alpha S_k && \\
        0 & = \frac{\beta_2}{N} I_2 \sum_{i=0}^r S_i - \gamma I_2&&\\
        0 & = \lambda \sum_{i=0}^r S_i  - \alpha V.
    \end{aligned}
    \right.
\end{equation}
Since $I_2 > 0$, we can solve the fourth equation of \eqref{SIR6-DE-x2} for $\displaystyle{\sum_{i=0}^r S_i}$, then substitute into the fifth equation of \eqref{SIR6-DE-x2} to solve for $V$. Substituting this value  into the first equation of \eqref{SIR6-DE-x2} and inductively using the the first three equations of \eqref{SIR6-DE-x2} to solve for $S_i$, $i=1, \ldots, k$, in terms of $I_2$ gives
\begin{align}
\label{SIR6-x2-eqs}
    \begin{cases}
        S_0 = \displaystyle{
                \frac{
                    \gamma I_2 \left(
                        \frac{\beta_2 I_2 + N(k\alpha + \lambda)}{N k\alpha}
                    \right)^{-r}
                  + \frac{N\lambda\gamma}{\beta_2}
                } {
                    \frac{\beta_2}{N} I_2 + \lambda
                }
            }
    \\  S_i = \displaystyle{
                \frac{\gamma I_2}{k\alpha}
                \left(
                    \frac{\beta_2 I_2 + N(k\alpha + \lambda)}{N k\alpha}
                \right)^{i-1-r}
            },                                        & i = 1, \ldots, r
    \\  S_i = \displaystyle{
                \frac{\gamma I_2}{k\alpha}
            },                                        & i = r+1 ,\ldots, k
    \\  \displaystyle{
            \sum_{i=0}^r S_i = \frac{N\gamma}{\beta_2},
        }
    \\  V = \displaystyle{
                \frac{N\gamma\lambda}{\beta_2\alpha}.
            }
    \end{cases}
\end{align}
We substitute the third through fifth equations in \eqref{SIR6-x2-eqs} into the population conservation to get
\begin{align*}
    N
 = I_2 + \sum_{i=0}^r S_i + \sum_{i=r+1}^k S_i + V
 = I_2 + \frac{N\gamma}{\beta_2}
      + \frac{\gamma(k-r)I_2}{k\alpha} + \frac{N\gamma\lambda}{\beta_2\alpha}
\end{align*}
which can be solved for $I_2$ to give
\begin{align*}
    I_2 = \frac{
        N k (\alpha \beta_2 - (\alpha + \lambda) \gamma)
    } {
        \beta_2 (k\alpha + (k-r) \gamma)
    }.
\end{align*}
We can substitute this into the first equation from \eqref{SIR6-x2-eqs} to get $S_0$, which we can substitute into the next generation matrix $FV^{-1}$ \eqref{x0-dfe-R1} to give
\begin{align*}
    \mathscr{R}_{12}
    & = \frac{\beta_1 S_0}{N \gamma}
      = \frac{\beta_1}{N \gamma} \left(  \frac{\gamma I_2
            \left(
                \frac{\beta_2 I_2 + N(k\alpha + \lambda)}{N k\alpha}
            \right)^{-r}
          + \frac{N\lambda\gamma}{\beta_2}
        } {
            \frac{\beta_2}{N} I_2 + \lambda
        } \right)
\\  & = \frac{\beta_1}{\beta_2} \left( \frac{
            (k\alpha (\beta_2 - \gamma) - \gamma \lambda k)
            \left( \frac{
                k\alpha ((k-r) \gamma + k\alpha)
            } {
                (k\alpha)^2
              + k\alpha ((k-r-1) \gamma + \beta_2 + \lambda)
              - \gamma \lambda r
            } \right)^r
          + \lambda ((k-r) \gamma + k\alpha)
        } {
            k\alpha (\beta_2 - \gamma + \lambda) - \gamma \lambda r
        } \right)
\end{align*}

\subsubsection{Strain 2 Competitive Reproduction Number ($\mathscr{R}_{21}$)}
To find the Strain 1 only equilibrium, we substitute $I_1 > I_2 = 0$ into \eqref{SIR6-DE} to get 
\begin{equation}
    \left\{ \; \; \;
    \label{SIR6-DE-x1}
    \begin{aligned}
        0 & = -\frac{\beta_1}{N} S_0 I_1 + k\alpha S_1 - \lambda S_0 + \alpha V&&  \\
        0 & = k\alpha {S}_{i+1} - k\alpha S_i - \lambda S_i, \; \; \; \; \; \; \; \; \; && i = 1, \ldots, r \\
        0 & = k\alpha {S}_{i+1} - k\alpha S_i, && i = r+1, \ldots, k-1 \\
        0 & = \gamma I_1 - k\alpha S_k, && \\
        0 & = \frac{\beta_1}{N} S_0 I_1 - \gamma I_1,&& \\
        0 & = \lambda \sum_{i=0}^r S_i  - \alpha V.
    \end{aligned}
    \right.
\end{equation}
Since $I_1 > 0$, the fourth equation in \eqref{SIR6-DE-x1} gives the value of $S_0$ which can be substituted into the first equation. The second through fourth equations of \eqref{SIR6-DE-x1} can be inductively used to give $S_i$ in terms of $I_1$ and the fifth equation can be used to solve $V$ in terms of $\sum_{i=0}^r S_i$. This gives
\begin{equation}
\label{SIR6-DE-1}
\begin{aligned}
    \begin{cases}
        S_0 = \displaystyle{
                \frac{N\gamma}{\beta_1}
            }
    \\  S_i = \displaystyle{
                \frac{\gamma I_1}{k\alpha}
                \left(
                    \frac{k\alpha + \lambda}{k\alpha}
                \right)^{i-1-r}
            },                                        & i = 1, \ldots, r
    \\  S_i = \displaystyle{
                \frac{\gamma I_1}{k\alpha}
            },                                        & i = r+1 ,\ldots, k
    \\  V = \displaystyle{
                \frac{\lambda}{\alpha} \sum_{i=0}^r S_i.
            }
    \end{cases}.
\end{aligned}
\end{equation}
Using the first two equations from \eqref{SIR6-DE-1} and the geometric series gives
\begin{equation}
    \label{S0}
\displaystyle{
            \sum_{i=0}^r S_i
                = \frac{\gamma}{\lambda} \left(
                    1
                  - \left(
                        \frac{k\alpha}{k\alpha + \lambda}
                    \right)^r
                \right) I_1
                + \frac{N\gamma}{\beta_1}
        }
\end{equation}
so that
\[V = \displaystyle{
                \frac{\gamma}{\alpha} \left(
                    1
                  - \left(
                        \frac{k\alpha}{k\alpha + \lambda}
                    \right)^r
                \right) I_1
                + \frac{N\lambda\gamma}{\beta_1\alpha}.
            }\]
Substituting these values into the population conservation equation gives
\begin{align*}
    N
    & = I_1 + \sum_{i=0}^r S_i + \sum_{i=r+1}^k S_i + V
\\  & = I_1
      + \frac{\gamma}{\lambda} \left(
            1
          - \left(
                \frac{k\alpha}{k\alpha + \lambda}
            \right)^r
        \right) I_1
      + \frac{N\gamma}{\beta_1}
      + \frac{(k-r)\gamma}{k\alpha} I_1
      + \frac{\gamma}{\alpha} \left(
            1
          - \left(
                \frac{k\alpha}{k\alpha + \lambda}
            \right)^r
        \right) I_1
      + \frac{N\lambda\gamma}{\beta_1\alpha}
\end{align*}
which can be solved for $I_1$ to get
\begin{equation}
\label{eq1}
\begin{aligned}
    I_1 = \frac{N k \lambda}{\beta_1} \left( \frac{
            (\beta_1 \alpha - (\alpha + \lambda) \gamma) (k \alpha + \lambda)^r
        } {
            (k \alpha \lambda + (k - r) \lambda \gamma + k (\alpha + \lambda) \gamma) (k \alpha + \lambda)^r
          - k (\alpha + \lambda) \gamma (k \alpha)^r
        } \right).
\end{aligned}
\end{equation}
Substituting \eqref{eq1} into \eqref{S0}, and then that into next generation matrix \eqref{x0-dfe-R2} gives
\[\begin{aligned}
    \mathscr{R}_{21}
    & = \frac{\beta_2}{\gamma N} \sum_{i=0}^r S_i
      = \frac{\beta_2}{\gamma N} \left(
            \frac{\gamma}{\lambda} \left(
                1
              - \left(
                    \frac{k\alpha}{k\alpha + \lambda}
                \right)^r
            \right) I_1
          + \frac{N\gamma}{\beta_1}
        \right)
\\  & = \frac{\beta_2}{\beta_1} \left( \frac{
             \alpha \beta_1 k(k \alpha)^r
          + (\gamma \lambda r -((\alpha + \gamma) \lambda + \alpha \beta_1) k)
            (\alpha k + \lambda)^r
        } {
            k \gamma(\alpha + \lambda)(k \alpha)^r
          + (\gamma \lambda r-((\alpha + 2 \gamma) \lambda + \alpha \gamma) k )
            (\alpha k + \lambda)^r
        } \right).
\end{aligned}\]

\section{Linear Chain Trick}
\label{sec:3}

In order to prove Theorem \ref{thm:3} and Theorem \ref{thm:4}, we need the following results. The proofs can be found in the indicated literature.

\begin{lemma}[Corollary of Leibniz Integral Rule, \cite{protter1985}]
    \label{leibniz}
    Suppose $f(t,x)$ and $\frac{\partial}{\partial t}f(t,x)$ are continuous everywhere and define $F(t) = \int_{- \infty}^t f(t,s) \; ds$. Then
    \[F'(t) = f(t,t) + \int_{-\infty}^t \frac{\partial}{\partial t} f(t,s) \; ds.\]
\end{lemma}

\begin{lemma}[Chapter 7, pp. 123-124, \cite{Smith2010}]
    \label{gamma}
    Consider the following series of Erlang distributions, with common rate parameter $k \alpha$ and the range of shape parameters $i= 1, \ldots, k$: 
\begin{equation}
    \label{erlang}
g^i_{k\alpha}(t) = \frac{(k\alpha)^it^{i-1}e^{-k \alpha t}}{(i-1)!}, \; i = 1 \ldots, k.
\end{equation}
The distributions \eqref{erlang} satisfy the following system of differential equations 
\begin{equation}
\label{eq:500}
    \left\{
    \begin{aligned}
        \frac{d}{dt} g_{k\alpha}^1(t) & = -k\alpha g_{k\alpha}^1(t), & & g_{k\alpha}^1(0) = k\alpha, & & \\
        \frac{d}{dt} g_{k\alpha}^i(t) & = k\alpha (g_{k\alpha}^{i-1}(t) - g_{k\alpha}^i(t)), & & g_{k\alpha}^i(0) = 0, \; \; \; & & i = 2, \ldots, k.
    \end{aligned}
    \right.
\end{equation}
\end{lemma}

\begin{proof}
For $i=1$, we have
\[g^1_{k\alpha}(t) = k\alpha e^{-k \alpha t}\]
so that
\[\frac{d}{dt} g^1_{k\alpha}(t) = \frac{d}{dt} \left[ k\alpha e^{-k \alpha t}\right] = -k\alpha \left[ k\alpha e^{-k \alpha t}\right] = - k\alpha g^1_{k\alpha}(t)\]
and 
\[g^1_{k\alpha}(0) = k\alpha e^{-k\alpha(0)} = k\alpha.\]

For $i =2, \ldots, k$, we have
\[\begin{aligned}
\frac{d}{dt} g^i_{k\alpha}(t) & = \frac{d}{dt} \left[ \frac{(k\alpha)^it^{i-1}e^{-k \alpha t}}{(i-1)!}\right]\\
& = \frac{(i-1)(k\alpha)^i t^{i-2}e^{-k \alpha t}}{(i-1)!}-\frac{(k\alpha)(k\alpha)^it^{i-1}e^{-k \alpha t}}{(i-1)!} \\
& = (k\alpha) \left[ \frac{(k\alpha)^{i-1} t^{i-2}e^{-k \alpha t}}{(i-2)!} \right] - (k\alpha) \left[ \frac{(k\alpha)^it^{i-1}e^{-k \alpha t}}{(i-1)!}\right] \\
& = k\alpha ( g^{i-1}_{k\alpha}(t) - g^i_{k\alpha}(t) )
\end{aligned}\]
and
\[g^i_{k\alpha}(0) = \frac{(k\alpha)^i(0)^{i-1}e^{-k \alpha (0)}}{(i-1)!} = 0\]
and we are done.
\end{proof}

\subsection{Proof of Theorem \ref{thm:3}}
\label{sec:31}

For completeness, we reproduce Theorem \ref{thm:3} of the main text here.

\begin{theorem*}
    The system of distributed delay differential equations \eqref{SIR3-DE} is equivalent to the system of ordinary differential equations \eqref{SIR4-DE} with the following variable substitutions:
\[
    \begin{aligned}
        S_0(t) & = S(t) \\
        S_i(t) & = \frac{1}{k\alpha} \int_{-\infty}^t \Gamma(\tau) e^{-\frac{\beta_2}{N} \int_{\tau}^t I_2(s) \; ds} e^{-\lambda(t-\tau)} g_{k\alpha}^{r-i+1}(t-\tau) \; d\tau, & & i = 1, \ldots, r\\
        S_i(t) & = \frac{1}{k\alpha} \int_{-\infty}^t [\lambda (S(\tau)+R(\tau)) + \gamma(I_1(\tau)+I_2(\tau))]g_{k\alpha}^{k-i+1}(t-\tau)\; d\tau & & i = r+1, \ldots, k\\
    \end{aligned}
\]
where $R(t) = \sum_{i=1}^r S_i(t)$.
\end{theorem*}

\begin{proof}

We start by differentiating the substitutions $S_i(t)$, $i=r+1, \ldots, k$. By Lemma \ref{leibniz}, we have the following:
\[
\begin{aligned}
    \frac{dS_i}{dt} & = \frac{d}{dt} \left[ \frac{1}{k\alpha} \int_{-\infty}^t [\lambda (S(\tau) + R(\tau)) + \gamma(I_1(\tau)+I_2(\tau))]g_{k\alpha}^{k-i+1}(t-\tau)\; d\tau \right]\\
    & = \frac{1}{k\alpha} (\lambda (S(t) + R(t)) + \gamma(I_1(t)+I_2(t)))[g_{k\alpha}^{k-i+1}(0)] \\ & \;\;\;\;+ \frac{1}{k\alpha} \int_{-\infty}^t (\lambda (S(\tau) + R(\tau))+\gamma(I_1(\tau)+I_2(\tau)))\left[ \frac{d}{dt} g_{k\alpha}^{k-i+1}(t-\tau)\right] \; d\tau
\end{aligned}
\]
For $i=k$, $\frac{d}{dt} g_{k\alpha}^{k-i+1}(t-\tau) = \frac{d}{dt} g_{k\alpha}^{1}(t-\tau)$ and $g_{k\alpha}^{k-i+1}(0) = g_{k\alpha}^1(0)$ so that, by the first row of \eqref{eq:500}, we have
\[
\begin{aligned}
    \frac{dS_k}{dt} & = \frac{1}{k\alpha}(\lambda (S(t) + R(t)) + \gamma(I_1(t) + I_2(t)))(k\alpha) \\ & \;\;\;\; - k\alpha \left[ \frac{1}{k\alpha} \int_{-\infty}^t (\lambda (S(\tau) + R(\tau)) + \gamma(I_1(\tau)+I_2(\tau))) g_{k\alpha}^{1}(t-\tau) \; d\tau \right] \\
    & = \lambda (S(t) + R(t)) + \gamma(I_1(t) + I_2(t)) - k\alpha S_k(t) \\
    & = \lambda \sum_{i=0}^r S_i(t) + \gamma(I_1(t) + I_2(t)) - k\alpha S_k.
\end{aligned}
\]
For $i=r+1, \ldots, k-1$, by the second row of \eqref{eq:500}, we have
\[ 
\begin{aligned}
    \frac{dS_i}{dt} &  = \frac{1}{k\alpha}(\lambda (S(t) + R(t)) + \gamma(I_1(t) + I_2(t))) ( 0 ) \\ & \;\;\;\; - k\alpha \left[ \frac{1}{k\alpha} \int_{-\infty}^t (\lambda (S(\tau) + R(\tau)) + \gamma(I_1(\tau)+I_2(\tau))) \left(g_{k\alpha}^{k-i+2}(t-\tau) - g_{k\alpha}^{k-i+1}(t-\tau)\right) \; d\tau \right] \\
    & = k\alpha \left( S_{i+1} - S_i \right).
\end{aligned}
\]

We now differentiate the substitutions $S_i(t)$, $i=1, \ldots, r$. By Lemma \ref{leibniz}, we have the following:
\[
\begin{aligned}
    \frac{dS_i}{dt} & = \frac{d}{dt} \left[ \frac{1}{k\alpha} \int_{-\infty}^t \Gamma(\tau) e^{-\frac{\beta_2}{N} \int_{\tau}^t I_2(s) \; ds} e^{-\lambda(t-\tau)} g_{k\alpha}^{r-i+1}(t-\tau) \; d\tau \right]\\
    & = \frac{1}{k\alpha} \Gamma(t) [g_{k\alpha}^{r-i+1}(0)] + \frac{1}{k\alpha} \int_{-\infty}^t \Gamma(\tau) \left[ \frac{d}{dt} e^{-\frac{\beta_2}{N} \int_{\tau}^t I_2(s) \; ds}\right] e^{-\lambda(t-\tau)} g_{k\alpha}^{r-i+1}(t-\tau) \; d\tau \\ & \;\;\;\;+  \frac{1}{k\alpha} \int_{-\infty}^t \Gamma(\tau) e^{-\frac{\beta_2}{N} \int_{\tau}^t I_2(s) \; ds} \left[ \frac{d}{dt} e^{-\lambda(t-\tau)}  \right] g_{k\alpha}^{r-i+1}(t-\tau) \; d\tau \\ & \;\;\;\;+ \frac{1}{k\alpha} \int_{-\infty}^t \Gamma(\tau) e^{-\frac{\beta_2}{N} \int_{\tau}^t I_2(s) \; ds} e^{-\lambda(t-\tau)} \left[ \frac{d}{dt} g_{k\alpha}^{r-i+1}(t-\tau) \right] \; d\tau\\
    & = \frac{1}{k\alpha} \Gamma(t) [g_{k\alpha}^{r-i+1}(0)] - \frac{\beta_2}{N} I_2(t) \left[ \frac{1}{k\alpha} \int_{-\infty}^t \Gamma(\tau) e^{-\frac{\beta_2}{N} \int_{\tau}^t I_2(s) \; ds} e^{-\lambda(t-\tau)} g_{k\alpha}^{r-i+1}(t-\tau) \; d\tau\right] \\ & \;\;\;\;- \lambda \left[ \frac{1}{k\alpha} \int_{-\infty}^t \Gamma(\tau) e^{-\frac{\beta_2}{N} \int_{\tau}^t I_2(s) \; ds} e^{-\lambda(t-\tau)} g_{k\alpha}^{r-i+1}(t-\tau) \; d\tau\right] \\ & \;\;\;\;+ \frac{1}{k\alpha} \int_{-\infty}^t \Gamma(\tau) e^{-\frac{\beta_2}{N} \int_{\tau}^t I_2(s) \; ds} e^{-\lambda(t-\tau)} \left[ \frac{d}{dt} g_{k\alpha}^{r-i+1}(t-\tau) \right] \; d\tau\\
    & = \frac{1}{k\alpha} \Gamma(t) [(g_{k\alpha}^{r-i+1}(0)] - \frac{\beta_2}{N} I_2(t) S_i(t) - \lambda S_i(t)\\ & \;\;\;\; + \frac{1}{k\alpha} \int_{-\infty}^t \Gamma(\tau) e^{-\frac{\beta_2}{N} \int_{\tau}^t I_2(s) \; ds} e^{-\lambda(t-\tau)} \left[ \frac{d}{dt} g_{k\alpha}^{r-i+1}(t-\tau) \right] \; d\tau.
\end{aligned}
\]
For $i=r$, $\frac{d}{dt} g_{k\alpha}^{r-i+1}(t-\tau) = \frac{d}{dt} g_{k\alpha}^{1}(t-\tau)$ and $g_{k\alpha}^{r-i+1}(0) = g_{k\alpha}^1(0)$ so that, by the first row of \eqref{eq:500}, we have
\[
\begin{aligned}
    \frac{dS_r}{dt} 
    & = \frac{1}{k\alpha}\Gamma(t) (k\alpha) - \frac{\beta_2}{N} I_2(t) S_r(t) - \lambda S_r(t) - k\alpha \left[ \frac{1}{k\alpha} \int_{-\infty}^t \Gamma(\tau) e^{-\frac{\beta_2}{N} \int_{\tau}^t I_2(s) \; ds} e^{-\lambda(t-\tau)}g_{k\alpha}^{1}(t-\tau) \; d\tau \right]\\
    & = \Gamma(t) - \frac{\beta_2}{N} I_2(t) S_r(t) - \lambda S_r(t) - k\alpha S_r(t) \\
    & = k \alpha \left[ \frac{1}{k \alpha} \int_{-\infty}^t (\lambda (S(\tau) + R(\tau)) + \gamma(I_1(\tau)+I_2(\tau))) g_{k\alpha}^{k-r}(t-\tau)\; d\tau \right] - \frac{\beta_2}{N} I(t) S_i(t) - \lambda S_r(t) - k\alpha S_r(t)\\
    & = k \alpha S_{r+1}(t) - \frac{\beta_2}{N} I(t) S_i(t) - \lambda S_r(t) - k\alpha S_r(t).
\end{aligned}
\]
For $i=1, \ldots, r-1$, by the second row of \eqref{eq:500}, we have
\[
\begin{aligned}
    \frac{dS_i}{dt} & = \frac{1}{k\alpha} \Gamma(t) (0) - \frac{\beta_2}{N} I_2(t) S_i(t) - \lambda S_r(t) \\ & \;\;\;\; + k\alpha \left[ \frac{1}{k\alpha} \int_{-\infty}^t \Gamma(\tau) e^{-\frac{\beta_2}{N} \int_{\tau}^t I_2(s) \; ds} e^{-\lambda(t-\tau)} \left( g_{k\alpha}^{r-i+2} - g_{k\alpha}^{r-i+1}(t-\tau) \right) \; d\tau \right] \\
    & = - \frac{\beta_2}{N} I_2(t) S_i(t) - \lambda S_i(t) + k\alpha \left( S_{i+1}(t) - S_i(t) \right).
\end{aligned}
\]
For $S_0$, we make the correspondence $S_0(t) = S(t)$ so that
\[
\begin{aligned}
    \frac{dS_0}{dt} = \frac{dS}{dt} & =  - \frac{\beta_1}{N} S(t) I_1(t) - \frac{\beta_2}{N} S(t) I_2(t) - \lambda S(t) + \Pi(t) \\
    & =  - \frac{\beta_1}{N} S_0(t) I_1(t) - \frac{\beta_2}{N} S_0(t) I_2(t)  - \lambda S_0(t) \\ & \;\;\;\; + k\alpha \left[ \frac{1}{k\alpha} \int_{-\infty}^t \Gamma(\tau) e^{-\frac{\beta_2}{N} \int_{\tau}^t I_2(s) \; ds} e^{-\lambda(t-\tau)} g_{k\alpha}^r(t- \tau) \; d\tau \right] \\
    & = - \frac{\beta_1}{N} S_0(t) I_1(t) - \frac{\beta_2}{N} S_0(t) I_2(t)  - \lambda S_0(t) + k\alpha S_1(t).
\end{aligned}
\]

We can also use the substitutions to derive the differential equations for $I_1(t)$, $I_2(t)$, and $R(t)$ in \eqref{SIR3-DE} from \eqref{SIR4-DE}. Since $S(t) = S_0(t)$ and $R(t) = \sum_{i=1}^r S_i(t)$, we have
\[ 
\left\{
\begin{aligned}
\frac{dI_1}{dt} & = \frac{\beta_1}{N} S_0(t)I_1(t) - \gamma I_1(t) = \frac{\beta_1}{N} S(t)I_1(t) - \gamma I_1(t) \\
\frac{dI_2}{dt} & = \frac{\beta_2}{N} \sum_{i=0}^r S_i(t) - \gamma I_2(t) = \frac{\beta_2}{N} (S(t)+R(t))I_2(t) - \gamma I_2(t).
\end{aligned}
\right.\]
For $R(t)$, noting that $k \alpha S_1(t) = \Pi(t)$ and $k\alpha S_{r+1}(t) = \Gamma(t)$, we have
\[
\begin{aligned}
\frac{dR}{dt} = \sum_{i=1}^r \frac{dS_i}{dt} & = \sum_{i=1}^{r} \left[ -\frac{\beta_2}{N} I_2(t) S_i(t) - \lambda S_i(t) + k\alpha \left( S_{i+1}(t) - S_i(t) \right) \right] \\ 
& = -\frac{\beta_2}{N} I_2(t) \sum_{i=1}^r S_i(t) - \lambda \sum_{i=1}^r S_i(t) + k\alpha \sum_{i=1}^r \left( S_{i+1}(t) - S_i(t) \right) \\
& = -\frac{\beta_2}{N} I_2(t) R(t) - \lambda R(t) + k\alpha S_{r+1}(t) - k\alpha S_1(t)\\
& = -\frac{\beta_2}{N} I_2(t) R(t) - \lambda R(t) + \Gamma(t) - \Pi(t)
\end{aligned}
\]
and we are done.
\end{proof}

\subsection{Proof of Theorem \ref{thm:4}}
\label{sec:32}

For completeness, we reproduce Theorem \ref{thm:4} here.
\begin{theorem*}
    The system of distributed delay differential equations \eqref{SIR5-DE} is equivalent to the system of ordinary differential equations \eqref{SIR6-DE} with the following variable substitutions:
\[
    \begin{aligned}
        S_0(t) & = S(t) \\
        S_i(t) & = \frac{1}{k\alpha} \int_{-\infty}^t \Gamma(\tau) e^{-\frac{\beta_2}{N} \int_{\tau}^t I_2(s) \; ds} e^{-\lambda(t-\tau)} g_{k\alpha}^{r-i+1}(t-\tau) \; d\tau, & & i = 1, \ldots, r\\
        S_i(t)  & = \frac{1}{k\alpha} \int_{-\infty}^t \gamma(I_1(\tau)+I_2(\tau)) g_{k\alpha}^{k-i+1}(t-\tau)\; d\tau, & & i = r+1, \ldots, k\\
    \end{aligned}
\]
where $R(t) = \sum_{i=1}^r S_i(t)$.
\end{theorem*}

\begin{proof}
The derivation of the equations for $S_i(t)$, $i=0, \ldots, k$, follow identically to that of those in the proof of Theorem \ref{thm:3}, with the exception that $\lambda = 0$ must be taken in the equations for $i=r+1, \ldots, k.$

The equations for $I_1(t)$, $I_2(t)$, and $R(t)$ are identical to those of Theorem \ref{thm:3}. It only remains to consider the equation for $V(t)$. We have
\[\frac{dV}{dt} = \lambda (S(t) + R(t)) - \alpha V(t) = \lambda \sum_{i=0}^r S_i(t)  - \alpha V(t)\]
since $S(t) = S_0(t)$ and $R(t) = \sum_{i=1}^r S_i(t)$, and we are done.
\end{proof}

\end{appendices}

\end{document}